\documentclass[12pt]{amsart}
\usepackage[normalem]{ulem}
\usepackage[utf8]{inputenc}
\usepackage{amsbsy,amssymb,amsfonts,amsthm,amsmath}
\usepackage{mathtools}
\usepackage[top=2cm,left=1.5cm,right=1.5cm,bottom=2cm]{geometry}
\usepackage{graphicx}
\usepackage{xcolor}
\usepackage{enumerate}
\usepackage[shortlabels]{enumitem}
\usepackage[bookmarks=true,hyperindex,pdftex,colorlinks,citecolor=red, linkcolor=blue]{hyperref}
\usepackage{float}
\usepackage{multicol}
\usepackage{dsfont}

\DeclareMathOperator{\Orb}{Orb}

\newtheorem{theorem}{Theorem}[section]
\newtheorem{lemma}[theorem]{Lemma}
\newtheorem{proposition}[theorem]{Proposition}
\newtheorem{corollary}[theorem]{Corollary}

{\theoremstyle{definition}
	\newtheorem{example}[theorem]{Example}
	\newtheorem{question}{Question}
	\newtheorem{remark}[theorem]{Remark}
    \newtheorem{notation}[theorem]{Notation}
}

\newcommand{\veps}{\varepsilon}

\def\R{\mathbb R}
\def\N{\mathbb N}

\def\C{\mathbb C}
\def\K{\mathbb K}

\def\U{\mathcal U}
\def\L{\mathcal L}
\def\H{\mathcal H}
\def\T{\mathcal T}
\def\D{\mathcal D}
\def\P{{\mathbb{P}}}
\def\Mes{\mathfrak{m}}

\address[Fernando Costa Jr.]{Universidade Federal da Paraíba - Campus I, Departamento de Matemática, Jardim Universitário, s/n, Bairro Castelo Branco, CEP 58051-900, João Pessoa, Brazil}
\email{fernando@mat.ufpb.br}

\address[Romuald Ernst]{Univ. Littoral Côte d’Opale, UR 2597, LMPA, Laboratoire de Mathématiques Pures et Appliquées Joseph Liouville, Centre Universitaire de la Mi-Voix, Maison de la Recherche Blaise-Pascal, 50 rue Ferdinand Buisson, BP 699, 62228 Calais, France}
\email{romuald.ernst@univ-littoral.fr}

\address[Monia Mestiri]{Université d’Artois, UR2462, Laboratoire de Mathématiques de Lens (LML), F-62300 Lens, France}
\email{monia.mestiri@univ-artois.fr}

\address[Augustin Mouze]{Ecole Centrale de Lille, Université de Lille, CNRS, UMR 8524 - Laboratoire Paul Painlevé, F-59000 Lille, France}
\email{augustin.mouze@univ-lille.fr}

\allowdisplaybreaks
\title{Existence and non-existence in common hypercyclicity and its frequent variants}

\author{F. Costa Jr., R. Ernst, M. Mestiri and A. Mouze}
\date{\today}

\begin{document}


\begin{abstract}
Given a family $\mathcal T = (T_\lambda)_{\lambda\in\Lambda}$ of bounded linear operators acting on the same $F$-space $X$ and indexed by a set of parameters $\Lambda \subset \mathbb{R}^d$, $d\geq1$, we explore the existence and non-existence of vectors that simultaneously satisfy, for all elements of $\mathcal T$, one of the following three dynamical properties: hypercyclicity, upper frequent hypercyclicity, and frequent hypercyclicity. Our approach relies on associating to $\mathcal T$ a ``Lipschitz constant function'' $F$, whose growth encodes the interaction between the operators and the parameter set $\Lambda$. We show that the asymptotic behavior of $F$ determines sharp thresholds between existence and non-existence of common vectors. In the hypercyclicity setting, we identify the divergence of $\sum 1/F(n)^s$, where $s$ is the Hausdorff dimension of $\Lambda$, as a natural condition for existence, and prove complementary non-existence results, which are optimal for weighted shifts.
For upper frequent hypercyclicity, we obtain criteria of non-existence in terms of series of the form $\sum 1/F(\gamma^n)$ for parameter sets in $d$-dimensions. Considered in the one-dimensional case, these criteria are optimal and allow us to contrast known existence results of Mestiri for logarithmic-type growth with non-existence beyond this scale. For frequent hypercyclicity, we establish a dichotomy showing that common vectors exist essentially only when $F$ is bounded, while any unbounded growth prevents their existence. Our results apply to a broad class of parameter sets, including self-similar fractals, providing a unified perspective on how growth conditions and geometric features of $\Lambda$ determine common dynamical behavior. Finally, in the specific context of weighted shifts, we prove that two operators do not necessarily share the same frequently hypercyclic vectors if the sequence of their weight products ratios admits two distinct non-zero cluster points, establishing the optimality of a result of Grivaux, Matheron and Menet.
\end{abstract}

\maketitle

\section{Introduction} In the field of linear dynamical systems, one studies the dynamical properties of systems $(X,T)$, where $X$ is a separable infinite-dimensional $F$-space (sometimes a Fréchet space or even a Banach space) and $T:X\to X$ is a continuous linear map, which we shall call \emph{operator}. The main property of the theory is that of hypercyclicity. We say that $u\in X$ is a \emph{hypercyclic vector} for the operator $T$ whenever its orbit $\Orb(u;T)=\{T^nu : n\in\N_0\}=\{u, Tu, T^2u, \ldots\}$ is a dense set in $X$, where $\N_0:=\N\cup\{0\}$ denotes the set of non-negative integers. The set of hypercyclic vectors for $T$ is denoted by $HC(T)$. When $HC(T)$ is a non-empty set, we say that $T$ is a \emph{hypercyclic operator}. In this case, it follows from Birkhoff Transitivity Theorem that $HC(T)$ is residual (see \cite[Theorem 9.20]{GottHedl}).

Consider a set of parameters $\Lambda$ in some topological space and suppose that $(T_\lambda)_{\lambda \in\Lambda}$ is a family of hypercyclic operators acting on the same Banach space $X$. If $\Lambda$ is countable, then a direct application of the Baire Category Theorem guarantees that 
\begin{equation}\label{eq:cap:HC}
    \bigcap_{\lambda\in \Lambda}HC(T_\lambda)\neq\varnothing,
\end{equation}
since the set $HC(T_\lambda)$ is residual for every $\lambda\in\Lambda$. 
However, if $\Lambda$ is uncountable, then finding an element in the intersection \eqref{eq:cap:HC} is a non-trivial problem. Such an element is called a \emph{common hypercyclic vector} for the family $(T_\lambda)_{\lambda\in\Lambda}$. If such a vector exists, we say that $(T_\lambda)_{\lambda\in\Lambda}$ is a \emph{common hypercyclic family}. The first positive result in this trend was obtained by Abakumov and Gordon \cite{AbGo1} and, independently, by Peris \cite{PerisCommon}. They both answered positively a question raised by Salas \cite{Salas} by showing that the family $(\lambda B)_{\lambda>1}$ of Rolewicz operators $B(x_n)_n\mapsto (x_{n+1})_n$, $n\in\N_0$, acting on $\ell^2$ is common hypercyclic. In the same year, Bayart \cite{bayart2004} generalized their ideas and obtained a more general result on the family $(\lambda T)_{\lambda>1}$ of multiples of a single operator $T$ acting on some Banach space $X$. However, it was only with the work of Costakis and Sambarino \cite{CosSam} that the first criterion for general families was obtained.

The Costakis-Sambarino Criterion \cite[Theorem 12]{CosSam} was originally obtained in the context of \emph{universality} on $F$-spaces. Let $X$ be an $F$-space and, for each $n\in\N_0$, let $T_n:X\to X$ be an operator, where $T_0=Id$ is the identity map on $X$. We say that $u\in X$ is a \emph{universal vector} for the family $(T_n)_{n\in\N_0}$ whenever its orbit $\Orb(u;(T_n)_{n\in\N_0})=\{T_nu : n\in\N_0\}$ is dense in $X$. In this case, the family $(T_n)_{n\in\N_0}$ is said to be \emph{universal}. Now let $\Lambda\subset \R_+:=(0,+\infty)$ be a $\sigma$-compact subset and let $\T=\{T_{\lambda,n} : \lambda\in\Lambda, n\in\N_0\}$ be a \emph{continuous family of operators}, that is, such that $(u,\lambda)\mapsto T_{\lambda,n}u$ is a continuous map from $X\times \Lambda$ into $X$, for all $n\in\N_0$. We also suppose that there exist a dense subset $\D\subset X$ and maps $S_{\lambda,n}:\D\to\D$ such that $T_{\lambda,n}S_{\lambda,n}=Id$ for all $\lambda\in\Lambda$ and $n\in\N_0$. The original criterion obtained by Costakis and Sambarino can be stated as follows.
\begin{theorem}[Costakis-Sambarino Criterion]
    Assume that, for all $u\in \D$ and all $K\subset \Lambda$ compact, the following properties hold true.
    \begin{enumerate}[label={(CS\arabic*)}]
        \item \label{CS:1} There exist $\kappa\in\N$ and a summable sequence of positive numbers $(c_k)_{k}$ such that, for all $\lambda,\mu\in K$,
     \begin{enumerate}[$(a)$]
            \item $\big\| T_{\lambda,n+k}S_{\mu,n} u\big\|\leq c_k$ for any $n\geq0, k\geq \kappa, \mu\leq \lambda$, 
            \item $\big\| T_{\lambda,n}S_{\mu,n+k} u\big\|\leq c_k$ for any $n\geq0, k\geq \kappa, \lambda\leq \mu$.
\end{enumerate}
\item \label{CS:2} Given $\eta>0$, one can find $\tau>0$ such that, for all $n\in\N$ and all $\lambda, \mu\in K,$
\[ |\lambda-\mu|\leq\frac{\tau}{n}\implies \|T_{\lambda,n}S_{\mu,n}u-u\|\leq\eta. \]
\end{enumerate}
Then the set of common universal vectors for $\T$ is a dense $G_\delta$ subset of $X$.
\end{theorem}
With this result at hand, many more common hypercyclic families were obtained, and in particular $(\lambda B)_{\lambda >1}$ on $c_0$ or $\ell^p, 1\leq p<+\infty$. Notice that the above criterion applies to families of operators indexed by parameter sets in $\R$. As observed by Costakis and Sambarino themselves, if one wants to get a result for sets of parameters in $\R^d$, $d\geq 2$, then we \emph{should} replace $\frac{\tau}{n}$ in condition \ref{CS:2} by $\frac{\tau}{n^{1/d}}$. This commentary probably concerns an example that appears in Abakumov and Gordon's paper \cite{AbGo1}, which shows that, if for some set of parameters $\Lambda\subset \R^2$ the family $(\lambda B\times \mu B)_{(\lambda,\mu)\in\Lambda}$ acting on $\ell^2\times \ell^2$ has a common hypercyclic vector, then $\Lambda$ must have null Lebesgue measure (this example is credited to Borichev). More precisely, as shown in \cite{BCMparam}, $\Lambda$ must have Hausdorff dimension at most 1. Aiming to reach more examples of families of operators, sometimes indexed by more intricate sets of parameters, the idea of changing $\frac{\tau}{n}$ for something else was explored in \cite{BCMparam} and \cite{CostaSelf}. 

Indeed, we may wonder what happens when one replaces $\frac{\tau}{n}$ by $\frac{\tau}{F(n)}$ in Costakis-Sambarino's result, where $F:\N\to\R$ is some increasing function. It is clear that this result remains valid if the growth rate of the function $F$ is reduced. On the other hand, we may wonder how much the growth rate of the function $F$ can be increased while keeping the theorem correct.
\begin{question}\label{Q1}
    For which growth rates of $F$ does Costakis and Sambarino's theorem remain valid when $\frac{\tau}{n}$ in~\ref{CS:2} is replaced by $\frac{\tau}{F(n)}$?
\end{question}
In practice, the condition 
\begin{equation}\label{Cond:Exist1}
    \exists \, C>0,\forall \, n\in\N, \forall \, \lambda, \mu\in K, \,
\|T_{\lambda,n}S_{\mu,n}u-u\|\leq C F(n)|\lambda-\mu|
\end{equation}
is stronger than \ref{CS:2}. Since it is easier to verify, in many examples authors check this condition rather than verifying \ref{CS:2} directly.

One may also be interested in non-existence results. In this direction, it seems natural to check at which point the approach used for positive results fails to work. 
Usually, obtaining common hypercyclic/universal vectors relies on a Baire Category argument, where we define an element $u$ as a sum with $p$ terms, say $u=u_1 + \cdots + u_p$, and where the parameter set $K$ is discretized into $p$ sub-intervals, say 
\begin{equation}
K = [\lambda_0,\lambda_1]\cup \ldots \cup [\lambda_{p-1},\lambda_p]
\end{equation}
(see the proof of the common hypercyclicity criterion \cite[Theorem 11.9]{grossebook}). Then, roughly speaking, for each $k=1,\dots,p$, we want $u_k$ to behave like a universal vector not only for $(T_{\lambda_i ,n})_n$ but also for any $(T_{\lambda,n})_n$ whenever $\lambda\in[\lambda_{k-1},\lambda_k]$. We expect this to happen if the interval $[\lambda_{k-1},\lambda_k]$ is small enough. Inside the Baire argument, we want to find $n_k$ such that $T_{\lambda,n_k}$ and $T_{\mu,n_k}$ are ``almost the same’’ for all $\lambda,\mu\in [\lambda_{k-1},\lambda_k]$. Let's say $n_k=k$ for simplicity. The ``almost the same’’ property can somehow be measured by saying that $T_{\lambda,k} S_{\mu,k}$ behaves like the identity if $\lambda$ and $\mu$ are close enough, which in turn happens if $\lambda,\mu\in [\lambda_{k-1},\lambda_k]$ and the interval $[\lambda_{k-1},\lambda_k]$ is very small. Now, on the contrary of \eqref{Cond:Exist1}, an inequality of the type
\begin{equation*}\label{intermediate-discussion2}
cF(k)|\lambda-\mu|\leq \|T_{\lambda ,k}S_{\mu , k}u-u\|
\end{equation*}
establishes a lower bound on ``how close’’ $T_{\lambda,k} S_{\mu,k}$ can be to the identity. The faster $F$ grows, the smaller $[\lambda_{k-1},\lambda_k]$ must be for $T_{\lambda ,k} S_{\mu ,k}$ to be close to the identity, and there must be an optimal growth rate where these smaller sub-intervals can or cannot cover $K$. 

To be able to compare it with existence results, a kind of condition that could appear in negative results reads as follows.
\begin{equation}\label{Cond:NonExist1}
    \forall \, u\in X, u\neq 0, \, \exists \, c>0, \forall \, n\in\mathbb{N}, \forall \, \lambda, \mu\in K, \,
c F(n)|\lambda-\mu|\leq \|T_{\lambda ,n}S_{\mu ,n}u-u\|.
\end{equation}
Indeed, with such a condition in non-existence results we could be able to compare the growth of $F$ allowing to apply positive and negative results in practice and study the limit growth that separates both of them. A similar discussion was carried out in \cite{BCMparam}, where the authors study the particular case $F(n)=n^\alpha$ with $\alpha>0$. The bigger $\alpha$ is, the faster $F$ grows. When $\Lambda\subset \R_+^d$ is a $d$-dimensional cube, the authors have obtained non-existence for $\alpha>1/d$, existence for $\alpha<1/d$ and the limit case $\alpha=1/d$ was later solved in \cite{CostaSelf} also for existence. Thus, the growth associated to $\alpha=1/d$ separates existence from non-existence in this context, therefore it is natural to call it \emph{optimal} or \emph{critical} in this setting.

\begin{question}\label{Q2}
    Suppose that Question~\ref{Q1} has been solved, can we find families of operators without common hypercyclic vectors satisfying (\ref{Cond:NonExist1}) for functions $F$ whose growth is as close as desired of the admissible growth given by Question~\ref{Q1}?
\end{question}
The answer to both Question~\ref{Q1} and \ref{Q2} would bring to light a critical growth separating existence results from non-existence ones. 

Almost all of the considerations about hypercyclicity that we discussed previously can also be applied to \emph{upper frequent hypercyclicicty}. 
Let $X$ be a Banach space and $T:X\to X$ be an operator. 
For a subset $A\subset\mathbb{N}_0$, we define the lower density $\underline {d}(A)$ and the upper density $\overline {d}(A)$ of $A$ by 
\[
\underline{d}(A):=\liminf_{n\rightarrow \infty}\frac{\#(A\cap[0,n])}{n+1}\quad\hbox{ and }\quad
\overline{d}(A):=\limsup_{n\rightarrow \infty}\frac{\#(A\cap[0,n])}{n+1},\]
where $\# E$ denotes the cardinality of a set $E\subseteq \N_0$. We say that $u\in X$ is a \emph{frequently hypercyclic vector} for $T$ if it satisfies the following: for all $U\subset X$ open and non-empty, 
\[ 
\underline{d}(\{k\in\mathbb{N}_0 : T^ku\in U\})>0.
\] 
We say that $u\in X$ is an \emph{upper frequently hypercyclic vector} for $T$ if it satisfies the property above with the upper density $\overline{d}$ in place of the lower density $\underline{d}$.
The sets of frequently and upper frequently hypercyclic vectors are denoted by $FHC(T)$ and $\U FHC(T)$, respectively. When $FHC(T)$  is non-empty, we say that $T$ is a \emph{frequently hypercyclic operator}. We define \emph{upper frequently hypercyclic operator} analogously. In 2015, Bayart and Ruzsa \cite{BayartRuzsa} have shown that, just like with $HC(T)$, the set $\U FHC(T)$ is either empty of residual.

An existence criterion on common upper frequent hypercyclicity  was first obtained in \cite{MestiriCommonUFHC} by Mestiri when the set of parameters is one-dimensional. 
It can be stated as follows. 
\begin{theorem}[{\cite[Theorem 7]{MestiriCommonUFHC}}]
\label{monia:thm} Let $\Lambda\subset \R$ be an interval, $X$ be a separable Fréchet space and $(T_\lambda)_{\lambda\in\Lambda}$ be a continuous family of operators acting on $X$. Suppose that, for every compact interval $K\subset \Lambda$, there exist some $\D\subset X$ dense and maps $S_{\lambda, n}:\D\to X$, $n\geq 0$, $\lambda\in K$, such that, for any $u\in\D$,
\begin{enumerate}[$(i)$]
    \item \label{monia:i} the series $\sum_{n=0}^m T^m_\lambda S_{\mu_n, m-n}u$ converges unconditionally and uniformly for all $m\geq 0$, all $\mu_0\geq \mu_1\geq \cdots \geq \mu_m$ in $K$ and all $\lambda\in K$;
    \item \label{monia:ii} the series $\sum_{n\geq0}T_\lambda^mS_{\mu_n,m+n}u$ converges unconditionally and uniformly for all $m\geq 0$, all $\mu_0\geq \mu_1\geq \cdots \geq \mu_m$ in $K$ and all $\lambda\leq \mu_0$ in $K$;
    \item \label{monia:iii} for any $\veps>0$, there exists some decreasing sequence $(d_n)_{n\geq 1}$ of positive numbers such that
    \begin{enumerate}[$(a)$]
        \item for any $n\ge 1$ and any $\lambda,\mu\in K$,
        \[0\leq \mu-\lambda\leq d_n\implies \|T_\lambda^nS_{\mu,n}u-u\|\leq \veps;\]
        \item\label{cond:gamma:monia} for any $\gamma\in\N$, the series $\sum_{k\geq 1}d_{\gamma^k}$ diverges;
    \end{enumerate}
    \item $(T_{\lambda}^nu)_{n\geq 0}$ converges uniformly to $0$ for all $\lambda\in K$.
\end{enumerate}
Then the set of common upper frequently hypercyclic vectors for $(T_\lambda)_{\lambda\in \Lambda}$ is residual in $X$.
\end{theorem}
If we look closely at condition \ref{monia:iii} in Theorem \ref{monia:thm}, the sequence $(d_n)_n$ that replaces $\frac{\tau}{n}$ in the Costakis-Sambarino Criterion is required to satisfy $\sum_{k\geq 1}d_{\gamma^k}=+\infty$ for all $\gamma\in\N$. This is satisfied if $d_n=\frac{\tau}{F(n)}$ with $F(n)\lesssim\log(n)$. In other words, the family $(T_\lambda)_{\lambda\in\Lambda}$ has a common upper frequently hypercyclic vector whenever $F$ grows as $\log$ or slower.
Moreover, Mestiri, in her doctoral thesis \cite{mestirithesis}, also obtains a one-dimensional result that uses a condition similar to \eqref{Cond:NonExist1} with $F(n)=n^\beta$, ensuring the non-existence of common upper frequently hypercyclic vectors. It is obvious that the two growth conditions on the function in the previous two theorems leave a gap (polynomial versus logarithmic growth).
\begin{question}\label{Q3}
    What is the critical growth rate for $F$ in the case of common upper frequent hypercyclicity when $\Lambda$ is an interval?
\end{question}

The previous discussion concerns one dimensional sets of parameters. However, as far as we know, the case of parameter sets having dimension greater than one has not been investigated yet for common upper frequent hypercyclicity.

\begin{question}\label{Q3ter}
    What is the critical growth rate for $F$ in the case of common upper frequent hypercyclicity for $\Lambda\subset \R^d$, $d>1$?
\end{question}
The study of common hypercyclicity have made a big leap forward with the work of Bayart, Menet and the first author \cite{BCMparam} in 2022. In their article, they have shed light on the deep link between the Hausdorff dimension of the set parameterizing the family and the existence and non-existence of common hypercyclic vectors for this family. Such a link has not been investigated for common upper frequent hypercyclicity until now. For example, it is known that the family $(\lambda B)_{\lambda>1}$ is common hypercyclic \cite{AbGo1} but not common upper frequently hypercyclic \cite{mestirithesis} even on compact intervals. On the other hand, reasoning exactly as in Proposition~7.7 from \cite{BayMath}, one may remark that for every Borel set $\mathcal{B}\subset (1,+\infty)$, there exists an uncountable set $\Lambda\subset \mathcal{B}$ such that the family $(\lambda B)_{\lambda\in\Lambda}$ shares a common upper frequently hypercyclic vector.
The next step would be to focus on intermediate sets between these ``small'' sets and the ``big'' compact intervals. One can for example think about the ternary Cantor set $\mathcal{C}_{[a,b]}$ on some compact interval $[a,b]$. 
\begin{question}\label{Q3bis}
    Is $\bigcap_{\lambda\in\mathcal{C}_{[a,b]}}\mathcal{U}FHC(\lambda B)$ empty? Does it depend on the chosen interval? 
\end{question}
As we have mentioned, the sets of hypercyclic and upper frequent hypercyclic vectors are either empty or topologically very big, which makes trivial the problem of finding a common hypercyclic or upper frequently hypercyclic vector when the set of parameters is countable. The set of \emph{frequently hypercyclic vectors}, however, is always meager (see \cite[Theorem 1]{MoothathuTwoRemarks}). Thus, even for countable families of operators, the problem of finding a common frequently hypercyclic vector is already non-trivial. This has been deeply explored in \cite{CEMM} with constructive general criteria of existence and criteria of non-existence in the case where the parameter set is countable. Indeed, since the set of frequently hypercyclic vectors is always meager, the Baire approach is not an option. Moreover, Grivaux, Matheron and Menet proved in \cite{GriMathMenorthogonality} that the ergodic approach is not applicable either, at least for weighted shifts. The question of the existence of common frequently hypercyclic vectors has also been addressed in \cite{BayGriMathMenHereditarilyFHC} in some particular cases with finitely many operators. For uncountable families, there exist very few results, which are structural. For example, a frequently hypercyclic version of Le\'on and Müller's Theorem asserts that the family $(\lambda T)_{\lambda\in\mathbb{T}}$ possesses a frequently hypercyclic vector, as proved in \cite{BayMath}. The underlying multiplicative group structure of $\T$ plays a crucial role in obtaining such a result. Other specific examples, also relying on group or semigroup structures, such as the family of translation operators on $H(\C)$, can be found in \cite{Bayhigh}. In a different vein, there exists another result restricted to the particular case of weighted shifts given in \cite{GriMathMenorthogonality} and due to Charpentier and Menet. However, the authors are not aware of any general criterion comparable to those known in the hypercyclic and upper frequently hypercyclic cases. This yields two questions.
\begin{question}\label{Q4}
    Does there exist a general criterion for uncountable families giving the existence of a common frequently hypercyclic vector?
\end{question}
If such a criterion exists, a condition of the kind of \eqref{Cond:Exist1} might appear. In that case, the following question probably seems natural at this point.
\begin{question}\label{Q5}
    Does there exist a critical rate of growth for $F$ separating the existence and the non-existence of common frequently hypercyclic vectors?
\end{question}
Throughout this article, we will address all of these questions. The growth of the function $F$ is at the core of our work.
The text is organized as follows.

In Section~\ref{sec:commonhc}, we explain how the increasing function $F:\N\to\R_+$ naturally appears as a ``Lipschitz constant'' in the study of common hypercyclicity results. In particular, we obtain a Costakis-Sambarino Criterion with $\tau/n$ replaced by $\tau/F(n)$, where the function $F$ may increase arbitrarily slowly, provided that the condition $\sum 1/F(n)=+\infty$ is satisfied. This answers Question~\ref{Q1}. Moreover, we provide new examples of families admitting common hypercyclic vectors for superlinear growth rate of $F$, which could not be reached before. 

In Section~\ref{sec:Common((U)F)HC:negative}, we turn to non-existence results of common hypercyclic, common upper frequently hypercyclic and common frequently hypercyclic vectors when $\Lambda\subset\R^d$, $d\geq 1$. Our first result concerns common hypercyclicity and we prove that the condition $\sum 1/F(n)=+\infty$ is necessary by giving a non-existence result in the setting of weighted shifts when $\Lambda$ has positive Hausdorff measure. This answers Question~\ref{Q2} in the case of shifts.
We also present a non-existence criterion for common upper frequently hypercyclic vectors and we prove that, for weighted shifts, there are no upper frequent hypercyclic vector when $\Lambda$ has finite and positive $s$-dimensional Hausdorff measure and $F$ grows strictly faster than $\log^{\frac{1}{s}}$. In the particular case where $d=1$ and $\Lambda$ is an interval, this shows that logarithmic growth is optimal, answering Question~\ref{Q3}. Moreover, our results also give an upper bound for the optimal growth when $\Lambda\subset\R^d$ with $d>1$. This partially answers Question~\ref{Q3ter}. Our result further allows to consider more irregular parameter sets than intervals, and in particular we prove that the family $(\lambda B)_{\lambda\in \mathcal{C}}$ has no upper frequently hypercyclic vector, where $\mathcal{C}$ is the ternary Cantor set on any compact interval, thus answering Question~\ref{Q3bis}. We conclude the section with a non-existence criterion for common frequent hypercyclicity inspired by the technics used in the upper frequently hypercyclic framework.

In Section~\ref{sec:CommonFHC:positive}, we answer Question~\ref{Q4} by establishing a general criterion for the existence of common frequently universal vectors for uncountable families of operators. As a particular case, we obtain an existence criterion for general families of operators indexed by an interval. Our results also yield new examples of uncountable families sharing a common frequently hypercyclic vector and demonstrate that, for common frequent hypercyclicity, the critical growth is given by bounded functions, answering Question~\ref{Q5}. This completes our study of admissible growth rates of $F$ in the three cases considered in this article. 
Furthermore, in the specific case of weighted shifts, a recent result by Grivaux, Matheron, and Menet ensures that, if the sequence of the ratios of the weight products of two frequently hypercyclic weighted shifts converges to a non-zero limit, then these shifts share the same frequently hypercyclic vectors \cite{GriMathMenorthogonality}. We show that this result is no longer necessarily true if the sequence of weight ratios has two non-zero cluster points, or even if this sequence of weight ratios converges along an infinite subset of $\mathbb{N}$.
Finally, in Section~\ref{sec:Conclusion} in light of the results obtained in this article, we discuss several open problems and future directions.

Throughout the paper, whenever $A$ and $B$ depend on some parameters, 
we will use the notation $A\lesssim B$ to assert that $A\leq CB$ with
some constant $C>0$ that does not depend on the involved parameters. Also, when we say that a property $\mathcal{P}$ is satisfied for all $n\gg 1$, we mean that there exists $n_0\in\N$ such that $\mathcal{P}$ is true for all $n\geq n_0$.

\section{Common hypercyclicity and Lipschitz constants} \label{sec:commonhc}

Before starting the upcoming discussion, let us define weighted backward shifts. Let $X\subset \K^{\N_0}$ be a Fréchet sequence space and $w=(w_n)_{n\in\N}$ be a sequence of non-zero scalars called \emph{weight}. The \emph{weighted backward shift} induced by $w$ is the linear map $B_w$ given by
\[B_w(x_0,x_1,x_2,\dots)= (w_1x_1,w_2x_2,w_3,x_3,\dots), \, \forall \, (x_n)_n\in X.\]
We say that $w$ is an \emph{admissible weight} when $B_w:X\to X$ is bounded. This class of operators is one of the most important and studied in the field of Linear Dynamics due to its flexibility and generality.

Since the article \cite{BaMacommon} from 2007 by Bayart and Matheron, for $\Lambda\subset\R$, the behavior of the terms $\log(w_n(\lambda))$, $n\in\N$, $\lambda\in\Lambda$, has been studied in great detail. Here, $\big(w(\lambda))_{\lambda\in\Lambda}$ is a family of admissible weights $w(\lambda)=\big(w_n(\lambda)\big)_{n\in\N}$
inducing a continuous family $(B_w(\lambda))_{\lambda\in\Lambda}$ of shift operators on a Fréchet sequence space $X$. In this case, we say that $\big(w(\lambda))_{\lambda\in\Lambda}$ is an \emph{admissible family of weights}. 
These studies lead to many interesting common hypercyclicity results for weighted shifts on spaces like $c_0$ or $\ell^p, 1\leq p <+\infty$.  
Sums of the terms $\log(w_n(\lambda))$, $n\in\N$, $\lambda\in\Lambda$, also appear in \cite{mestirithesis} and \cite{BCPdisjoint}. In \cite{CostaCommonAlg}, this approach is simplified by defining functions
\begin{equation} \label{def:fn}
    \begin{array}{rcl}
        f_n: \Lambda & \longrightarrow & \R  \\
        x & \longmapsto & f_n(x)=\sum_{k=1}^n \log\big(|w_k(x)|\big),
    \end{array}
\end{equation}
for each $n\in\N$,
and studying their Lipschitz behavior. In \cite{BCMparam} this idea is further developed. Indeed, taking $I$ an interval and assuming the existence of a non-decreasing map $F:\N\to\R_+$ and constants $c,C>0$ such that the weights $\big(w(x))_{x\in I}$ satisfy
	\begin{equation}\label{eq:cF:fn:CF}
	    cF(n)|x -y|\leq |f_n(x)-f_n(y)|\leq CF(n)|x-y|,
	\end{equation}
for all $n\in\N$ and $x,y\in I$, the authors provide a full characterization on the common hypercyclicity of the family $B_{w(x_1)}\times \overset{d}{\cdots} \times B_{w(x_d)}:X \times \overset{d}{\cdots} \times X \to X \times \overset{d}{\cdots} \times X$ of products of weighted backward shifts, now parametrized by $\Lambda=I^d$. 
One of the conditions that should be highlighted in the aforementioned characterization is the fact that the set of parameters $\Lambda$ has to be covered by a collection of balls of radius $\tau/F(n_k)$, for some arbitrarily small $\tau>0$ and some suitable increasing sequence $(n_k)_k$. This is the multi-dimensional version of Costakis and Sambarino's approach. In particular, when $\Lambda$ is an interval, it seems natural to request that the series $\sum_{n\in\N}1/F(n)$ diverges. Indeed, if this series converges to some number $s>0$ then, no matter the sequence $(n_k)_{k}$ chosen, it is impossible to cover $\Lambda$ by sub-intervals of size $\tau/F(n_k)$ whenever $\tau s<m(\Lambda)/2$. 
	On the other hand, if $\sum_{n\in\N}1/F(n)$ diverges and $(n_k)_k$ does not grow too fast, then $\sum_{k\in\N}1/F(n_k)$ still diverges and  we can expect to cover $\Lambda$ by intervals of size $\tau/F(n_k)$ no matter how small $\tau>0$ is (see Figure \ref{fig:covering:tauF}).

    \begin{figure}[H]
        \centering
        \includegraphics[width=0.7\linewidth]{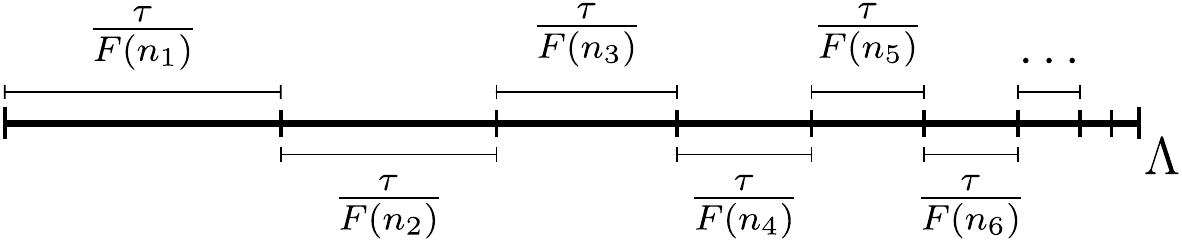}
        \caption{Covering of $\Lambda$ by parts of size $\frac{\tau}{F(n_k)}$, $k\in\N$}
        \label{fig:covering:tauF}
    \end{figure}

The condition $\sum_{n\in\N}1/F(n)=+\infty$ naturally appears in the 
Costakis-Sambarino criterion for common hypercyclicity. 
	Indeed, condition \ref{CS:2} is nothing but $\tau/F(n)$ with $F(n)=n$, which is the case of many typical applications such as the family $(\lambda B)_{\lambda>1}$ of Rolewicz operators on $c_0$ or the family $(\lambda D)_{\lambda>0}$ of multiples of the MacLane operator $D:f\mapsto f'$ on $H(\C)$. It is interesting to note that nowhere in the literature has the term $\tau/n$ been replaced by $\tau/F(n)$ for a general function $F$ to obtain a more general criterion. This is exactly what the following result does.

\begin{theorem} \label{CS:crit:F} Let $\Lambda\subset (0,+\infty)$ be a $\sigma$-compact set and let $\T=\{T_{\lambda ,n}:n\in\N_0, \lambda \in \Lambda\}$ be a continuous family of operators acting on an $F$-space $X$. Assume that there are $F:\N\to\R_+$ non-decreasing with $\sum_{n\in\N}\frac{1}{F(n)}=+\infty$, $\D\subset X$ dense and operators $S_{\lambda ,n}:\D\to \D, n\in\N_0,\lambda\in \Lambda,$ such that $T_{\lambda ,n}\circ S_{\lambda ,n}=Id$ and, for all $u\in \D$ and $K\subset \Lambda$ compact, the following properties hold true. 

\begin{enumerate}[$(i)$]
\item \label{CS:F:1} There exist $\kappa\in\N$ and a summable sequence of positive numbers $(c_k)_{k\in\mathbb{N}}$ such that, for all $\lambda,\mu\in K$, 
     \begin{enumerate}[$(a)$]
            \item $\| T_{\lambda ,n+k}S_{\mu ,n} u\|\leq c_k$ for any $n\geq0, k\geq \kappa, \mu\leq \lambda$, 
            \item $\| T_{\lambda ,n}S_{\mu ,n+k} u\|\leq c_k$ for any $n\geq0, k\geq \kappa, \lambda\leq \mu$.
\end{enumerate}
\item \label{CS:F:2} Given $\eta>0$, one can find $\tau>0$ such that, for all $n\gg 1$ and all $\lambda, \mu\in K,$
\[ 0\leq \mu-\lambda\leq\frac{\tau}{F(n)}
\implies \|T_{\lambda ,n}S_{\mu ,n}u-u\|\leq\eta. 
\]
\end{enumerate}
Then the set of common universal vectors for $\T$ is a dense $G_\delta$ subset of $X$.
\end{theorem}
\begin{proof}
We can suppose $\Lambda=\bigcup_{s\in I} K_s$, where the set $I$ is at most countable, and $K_s=[a_s,b_s]$ for all $s\in I$. Let $(V_t)_{t\in\mathbb{N}}$ be a countable basis of open sets for the topology of $X$. For each $s\in I$ and $t\in\mathbb{N}$, we define \[A(s,t)=\{u\in X : \forall \, \lambda\in K_s,\exists  \, m\in\N_0\text{ such that }T_{\lambda ,m}u\in V_t\}.\]
Then, of course, any element of $\bigcap_{s\in I}\bigcap_{t\in \N}A(s,t)$ is a universal vector for $\T$. We aim to show that each set $A(s,t)$ is open and dense.

That $A(s,t)$ is open is clear. We proceed to prove that it is dense. Given an arbitrary $U\subset X$ open and non-empty, fix $u_0\in U\cap\D$ and $v_t\in V_t\cap\D$. Choose $\eta>0$ small enough so that $B_{\|\cdot\|}(u_0,\eta)\subset U$ and $B_{\|\cdot\|}(v_t,3\eta)\subset V_t$, where we define, for $v\in X$ and $\varepsilon>0$, $B_{\|\cdot\|}(v,\varepsilon):=\{u\in X;\Vert v-u\Vert<\varepsilon\}$. Let $\kappa>0$ and $(c_k)_k$ such that property \ref{CS:F:1} holds for $u$ and $v_t$. We find $\tau>0$ and $N_0\in\N$ such that property \ref{CS:F:2} is valid for $v_t$ for all $n\geq N_0$. We take $N\geq N_0$ big enough so that $\sum_{k\geq N} c_k < \eta$ and we define $n_i=iN$, $i\geq 1$. Consider the partition of $K_s$ defined inductively by $\lambda_0=a_s$ and $\lambda_i=\lambda_{i-1}+\frac{\tau}{F(iN)}$. Since $F$ is non-decreasing, the series $\sum_{i\in\N}\frac{\tau}{F(iN)}$ still diverges. Thus, there exists $q\geq 1$ such that $\lambda_q\geq b_s$. We reset $\lambda_{q-1}< b_s=\lambda_q$ so that $K_s=\bigcup_{i=1}^q\Lambda_i$, where $\Lambda_i=[\lambda_{i-1},\lambda_i]$, $i=1,\dots,q$. We define $u=u_0+\sum_{i=1}^qS_{\lambda_i ,n_i}v_t$ and we show that $u\in A(s,t)\cap U$. Applying property \ref{CS:F:1}(b) with $n=0$ 
we get
\[\|u-u_0\|\leq \sum_{i=1}^q \|S_{\lambda_i , n_i}v_t\|\leq \sum_{i=1}^q c_{n_i}\leq \sum_{k\geq N}c_k<\eta,\]
which implies that $u\in U$. Now, given $\lambda\in K,$ there is $i_0\in\{1,\dots,q\}$ such that $\lambda\in \Lambda_{i_0}$. We choose $m=n_{i_0}$ and we apply properties \ref{CS:F:1} and \ref{CS:F:2} in order to obtain
\begin{align*}
    \|T_{\lambda ,m}u-v_t\|
        &=\bigg\|T_{\lambda ,n_{i_0}}u_0 + \sum_{i=1}^{q}T_{\lambda ,n_{i_0}}S_{\lambda_i ,n_i}v_t-v_t\bigg\|\\
        &\leq\|T_{\lambda ,n_{i_0}}u_0\| + 
        \sum_{i=1}^{i_0-1}\|T_{\lambda ,n_{i_0}}S_{\lambda_i ,n_i}v_t\|+ 
        \sum_{i=i_0+1}^{q}\|T_{\lambda ,n_{i_0}}S_{\lambda_i ,n_i}v_t\| + 
        \|T_{\lambda ,n_{i_0}}S_{\lambda_{i_0},n_{i_0},}v_t-v_t\|\\
        &\leq c_{i_0N}+\sum_{i=1}^{i_0-1} c_{(i_0-i)N}+\sum_{i=i_0+1}^{q}c_{(i-i_0)N}+\eta\\
        &\leq 2\sum_{k\geq N}c_k+\eta<3\eta,
\end{align*}
what implies $T_{\lambda ,m}u\in V_t$, that is, $u\in A(s,t)$. This concludes the proof.
\end{proof}

Unlike the classical Costakis-Sambarino criterion, the above result applies to families for which the associated function $F:n\mapsto F(n)$ has a more intricate behavior. This allows us to get new interesting applications. 

Recall that, in \cite{BCMparam}, the authors consider families of weights satisfying $w_1(a)\ldots w_n(a) = \exp(a n^{\alpha})$, where the parameter $\alpha$ is fixed. The following example changes this perspective and fixes the parameter $a$ instead.

\begin{example}
    Consider the family $(B_{w(\lambda)})_{\lambda\in (0,1]}$ of backward shifts induced by the family of weights $(w(\lambda))_{\lambda\in (0,1]}$ given by $w_1(\lambda)\ldots w_n(\lambda)=\exp(n^\lambda)$ and acting  on $X=c_0$ or $\ell^p, 1\leq p<+\infty$. Then $(B_{w(\lambda)})_{\lambda\in (0,1]}$ has a dense $G_\delta$-set of common hypercyclic vectors. Indeed, let us verify that this family satisfies Theorem \ref{CS:crit:F} with $F(n)=n\log(n)$. 
Let $0<a<b$ in $(0,1]$ and consider the family $\big(B_{w(\lambda)}\big)_{\lambda\in [a,b]}$. Let $\lambda,\mu\in [a,b]$. Assume that $\lambda\leq\mu$. 
We get 
    \[
    \vert e^{n^{\lambda}-n^{\mu}}-1\vert=1-e^{-(n^{\mu}-n^{\lambda})}\leq n^{\mu}-n^{\lambda}=n^{\lambda}(n^{\mu-\lambda}-1).
    \]
Using the inequality $e^t-1\leq te^t$ for $t\geq 0$, we write $n^{\mu-\lambda}-1\leq (\mu-\lambda)\log(n)n^{\mu-\lambda}$. We deduce
    \[
    \vert e^{n^{\lambda}-n^{\mu}}-1\vert\leq n^{\mu}\log(n)(\mu-\lambda)\leq n\log(n)(\mu-\lambda).
    \] 
    Thus, condition \ref{CS:F:2} of Theorem \ref{CS:crit:F} is satisfied. It remains to verify conditions \ref{CS:F:1} of Theorem \ref{CS:crit:F}. As usual with backward shifts, we define $(S_{w(\lambda)})_{\lambda\in (0,1]}$ as the family of forward shifts given by $S_{w(\lambda)}(e_j)=\frac{e_{j+1}}{w_{j+1}(\lambda)}$, where $(e_j)$ is the canonical unit sequence. Clearly we have $B_{w(\lambda)}^n\circ S_{w(\lambda)}^n=Id$ and $B_{w(\lambda)}^{n+k}\circ S_{w(\mu)}^n(e_j)=0$ for $k>j$. Finally, for $\lambda\leq\mu$, the inequality 
    \[
    \Vert B_{w(\lambda)}^{n}\circ S_{w(\mu)}^{n+k}(e_j)\Vert =\frac{\hbox{exp}((j+k+n)^{\lambda})}{\hbox{exp}((j+k)^{\lambda})}\frac{\hbox{exp}(j^{\mu})}{\hbox{exp}((j+k+n)^{\mu})}\lesssim e^{-k}
    \]
    easily allows us to verify conditions \ref{CS:F:1} of Theorem \ref{CS:crit:F}. The conclusion now follows.
\end{example}

\begin{remark}
    Notice that, in the previous example, $B_{w(\lambda)}$ is not continuous for $\lambda>1$.
\end{remark}

\begin{remark} \label{diago}
If we revisit the family of weights satisfying $w_1(a,b) \dots w_n(a,b) = \exp(an^{b})$ considered in \cite{BCMparam}, which served as the basis for the previous example, but allow this time for both parameters $a$ and $b$ to vary, one may wonder whether the family $(B_{w(a,b)})_{(a,b)\in(0,1]^2}$ has a dense $G_\delta$-set of common hypercyclic vectors. Neither the preceding theorem nor the results in \cite{BCMparam} appear to be applicable here. It is worth noting that, by restricting the analysis to the diagonal, i.e. $a = b$, a proof similar to the previous example yields a positive answer in this specific case, based on the following estimate: for $0<a<b$ in $(0,1]$ and $\lambda,\mu\in [a,b]$, with $\lambda\leq\mu$, 
    \[
    \begin{array}{rcl}\vert e^{\lambda n^{\lambda}-\mu n^{\mu}}-1\vert=1-e^{-(\lambda n^{\lambda}-\mu n^{\mu})}&\leq& \mu n^{\mu}-\lambda n^{\lambda}\\&\leq&(\mu-\lambda)n^{\mu}+\lambda(n^{\lambda}-n^{\mu})\\&\leq& (\mu-\lambda)n^{\mu}+\lambda(\mu-\lambda)n^{\mu}\log(n)\\&\leq& (\mu-\lambda)(n+n\log(n)).
    \end{array}
    \]
\end{remark}

The only constraint on $F$ in Theorem \ref{CS:crit:F} is that $\sum_{n}1/F(n)=+\infty$. Therefore, we can get common hypercyclic vectors even when this series ``barely'' diverges, that is, diverges as slowly as one wants. 

\begin{example}
    Let $\Lambda\subset \R_+$ and $(w(\lambda))_{\lambda\in \Lambda}$ be a continuous family of weights satisfying \eqref{eq:cF:fn:CF} with $F(n)=n\big(\log^{(k)}(n)\big)^\beta$ for some $k\in \N$ and $\beta\leq1.$ Since $\sum_{n\in\N}1/F(n)=+\infty$, it follows that $\big(B_{w(\lambda)}\big)_{\lambda\in\Lambda}$ has a common hypercyclic vector. 
\end{example}

Having ``barely'' some kind of behavior or some type of growth is a recurrent theme in this article. For instance, as we shall see in Theorem 3.1, having $\sum_{n}1/F(n)<+\infty$ implies the non-existence of common hypercyclic vectors for shifts. Therefore, at least for this class, condition $\sum_n1/F(n)=+\infty$ in Theorem \ref{CS:crit:F} is optimal. In the following sections, we shall obtain optimal conditions for common frequent and upper frequent hypercyclic vectors as well.

\begin{remark}
    We finish this section with an important observation. Although we have obtained multi-dimensional versions of the Costakis-Sambarino criterion in the recent literature, the formulation relies on specific types of functions $F:n\mapsto F(n)$. Indeed, in more than one dimension, the construction of the partitions becomes a highly non-trivial task. Thus, obtaining a multi-dimensional result for general $F$ is a very difficult problem (see Question \ref{Q7:CS:general} for a simple version of this open problem).
\end{remark}

\section{Non-existence results}\label{sec:Common((U)F)HC:negative}

In this section, we aim to obtain sufficient conditions for the non-existence of common vectors. It is worth mentioning that the results obtained are in a way optimal and represent the exact requirements in terms of growth on the associated increasing function $F:\N\to\R_+$ that allows us to get our results. Let us first discuss these ideas.

Let $X$ be a separable $F$-space and consider on $\R^d$ the $\sup$ norm. Let $d\geq1$ and consider a compact set of parameters $\Lambda\subset \R^d$ and a continuous family of hypercyclic operators $(T_\lambda)_{\lambda \in \Lambda}$ acting on $X$. 

As we mentioned in the introduction of this article, when $\Lambda$ is countable we immediately get a common hypercyclic vector through a Baire argument. On the other hand, if $\Lambda$ is uncountable, then one cannot apply the Baire theorem, at least not directly. Indeed, what we try to do in practice is to find a \emph{discretization} of $\Lambda$ with special properties, generally in the form of a covering with extra features. This is done in a way that allows us to apply the Baire argument. As we mentioned in Section \ref{sec:commonhc}, it is condition \ref{CS:F:2} of Theorem \ref{CS:crit:F} (see its proof) that allows us to obtain this discretization for one-dimensional sets of parameters $\Lambda\subset\R$, which is formed by intervals of size $\tau/F(n_k)$ for some sequence $(n_k)$. Of course, the slower $F$ grows, the ``easier'' it is to cover $\Lambda$. As we shall see, how slow $F$ has to be depends on the property. For hypercyclicity and $d=1$, $\sum 1/F(n)=+\infty$ is sufficient and, at least for shifts, it is also necessary. For general families of operators, it is hard to draw a conclusion by looking at condition \ref{CS:F:2} of Theorem \ref{CS:crit:F}, but still it can be done under some assumptions. It goes as follows.

As the reader can notice, in our existence results of common vectors for hypercyclicity, frequent and upper frequent hypercyclicity, essentially the same condition as \ref{CS:F:2} in Theorem \ref{CS:crit:F} appears (see condition \ref{monia:iii} of Theorem \ref{monia:thm} and condition \eqref{c4} of Theorem \ref{thmgeneral}). Assuming that the operators are invertible and taking $v=S_\mu^nu$, this condition takes the form 
\begin{equation} \label{cond:comp-1}
    \forall \, u\in X, \, \forall \, \veps>0, \, \exists \, \tau>0, \, \forall \, n\in\N, \, \forall \, \lambda,\mu\in\Lambda, \, \|\lambda-\mu\|_\infty<\frac{\tau}{F(n)}\implies \|T_{\lambda}^{n}v-T_{\mu}^{n}v\|<\veps.
\end{equation}
In practice, it is easier to check the following handier sufficient condition
\begin{equation} \label{cond:comp-1_1}
    \forall \, v\in X, \,\exists \, C>0, \, \forall \, n\in\N, \, \forall \, \lambda,\mu\in\Lambda, \, \|T_{\lambda}^{n}v-T_{\mu}^{n}v\|<C F(n)\| \lambda-\mu\|_\infty.
\end{equation}
Indeed, one tries to estimate the left-hand side of this inequality in terms of $\|\lambda-\mu\|_\infty$, and the coefficient depending on $n$ that appears gives the increasing function $F:\N\to\R_+$ associated to $(T_\lambda)_{\lambda\in\Lambda}$.

On the other hand, for non-existence results (see Theorems~\ref{modif_teo:non-ex} and \ref{teo:non-exFHC}), we find a condition of the form (we use here the Lebesgue measure): for all $u\in X$, there are $V\subset X$ open and non-empty and $\tau>0$ such that, for all $n\in\N$,
\begin{equation} \label{eq:discussion:measure}
    m(\{\lambda\in\Lambda : T_\lambda^nu\in V\})\leq \frac{\tau}{F(n)}.
\end{equation}
The intuition is that the greater $F$ can be taken, the less common vectors we have.
Under the hypothesis that the sets $\{\lambda\in\Lambda: T_{\lambda}^{n}u\in V\}$ are intervals (which is often but not always the case) and choosing $V=B(w,\frac{\varepsilon}{2})$, the inequality \eqref{eq:discussion:measure} is a consequence of the implication
\[\Vert T_{\lambda}^{n}u-T_{\mu}^{n}u\Vert<\varepsilon \Longrightarrow\| \lambda-\mu\|_\infty \leq \frac{\tau}{F(n)}.\]
Therefore, in this framework, the condition that must be satisfied for non-existence results can be understood as
\begin{equation} \label{cond:comp-2}
    \forall \, u\in X\setminus\{0\}, \, \forall \, \veps>0, \exists \, \tau>0, \, \forall \, n\in\N, \, \forall \, \lambda,\mu\in\Lambda, \, \| T_{\lambda}^{n}u-T_{\mu}^{n}u\|<\varepsilon \Longrightarrow\Vert \lambda-\mu\Vert_{\infty} \leq \frac{\tau}{F(n)}.
\end{equation}
Again, in practice it suffices to check the inequality
\begin{equation} \label{cond:comp-3}
\forall \, u\in X\setminus\{0\}, \, \exists \, c>0, \, \forall \, n\in\N, \, \forall \, \lambda,\mu\in\Lambda, \, cF(n)\|\lambda-\mu\|_\infty\leq \|T_{\lambda}^{n}u-T_{\mu}^{n}u\|.
\end{equation}

Conditions \eqref{cond:comp-1_1} and \eqref{cond:comp-3} support our claim that conditions of the type of condition \ref{CS:F:2} of Theorem~\ref{CS:crit:F}, \ref{monia:iii} of Theorem~\ref{monia:thm}, and condition \eqref{c4} of Theorem~\ref{thmgeneral} or \eqref{modif_teo:non-ex:eq} of Theorem~\ref{modif_teo:non-ex} and \eqref{teo:non-exFHC:eq} of Theorem~\ref{teo:non-exFHC} cannot really be improved in terms of the function $F$.

\subsection{Hypercyclicity} \label{sec:hc:diverge}

As we have seen, condition \ref{CS:F:2} in Theorem \ref{CS:crit:F} is an important property that one has to verify for the existence of common hypercyclic vectors for a continuous family $(T_\lambda)_{\lambda\in\Lambda}$ acting on an $F$-space $X$ and parametrized by a $\sigma$-compact set $\Lambda\subset \R$. 
Of course, if $\Lambda$ is compact, it follows by the (uniform) continuity of the family that \ref{CS:F:2} of Theorem \ref{CS:crit:F} will be verified if $F(n)$ is big enough. However, if we choose its value too big, then $\sum_{n}1/F(n)<+\infty$ and we cannot conclude from the Costakis-Sambarino criterion. When it comes to backward shifts $T_\lambda = B_{w(\lambda)}$, we know that, for each $n\in\N$, $F(n)$ is the Lipschitz constant of $f_n$ defined in \eqref{def:fn}. 
The next result shows that condition $\sum_{n}1/F(n)=+\infty$ is optimal in the context of \eqref{eq:cF:fn:CF}. 
Indeed, we can obtain this conclusion as an application of \cite[Theorem 3.1]{BCMparam} as we shall do right now.

Recall that, if $\phi:(0,+\infty)\rightarrow (0,+\infty)$ is a \emph{gauge function}, i.e. a non-decreasing function satisfying $\lim_{t\rightarrow 0^+}\phi(t)=0$, then the \emph{$\phi$-Hausdorff outer measure} of a set $\Lambda\subset\mathbb{R}^d$ is defined as
\[
\H^\phi(\Lambda)=\lim_{\varepsilon\rightarrow 0}\inf_{R\in R_{\varepsilon}(\Lambda)}\sum_{B\in R}\phi(\hbox{diam}(B)),
\]
where $R_{\varepsilon}(\Lambda)$ is the set of all countable coverings of $\Lambda$ by balls $B$ of diameter $\hbox{diam}(B)\leq \varepsilon$.

\begin{theorem}\label{thm:F:conv:non-ex}
		Let $X=c_0$ or $\ell^p, 1\leq p<+\infty$, $I\subset\mathbb{R}$ and $d\in\N$. 
        Consider an admissible family of weights $(w(x))_{x\in I}$ such that there are $F:\N\to\mathbb{R}_+$ and $c>0$ satisfying \[
        |f_n(x)-f_n(y)|
        \geq cF(n)|x-y|\]
		for all $x,y\in I$ and $n\in \N$. If for some $s>0$ we have $\sum_{n\in\N}1/F(n)^s<+\infty$, then, for any $\Lambda\subset I^d$ with positive $s$-dimensional Hausdorff outer measure, we have $\bigcap_{\lambda\in\Lambda}HC(B_{w(\lambda(1))}\times\cdots\times B_{w(\lambda(d))})=\varnothing$.
	\end{theorem}
	\begin{proof}
		We aim to apply \cite[Theorem 3.1]{BCMparam} with $\psi=4^{-1}cF$, $\phi(x)=x^s$, $v=(\overset{d}{e_0,\dots,e_0})$ and $\delta=1/2$. Let $\lambda,\mu\in\Lambda$ and $u\in X^d$ such that 
		\[\|(B_{w(\lambda(1))}\times\cdots\times B_{w(\lambda(d))})u - v\|<\frac{1}{2} \quad \text{and} \quad \|(B_{w(\lambda(1))}\times\cdots\times B_{w(\lambda(d))})^nu - v\|<\frac{1}{2}.\]
Then, for each $1\leq k\leq d$, we have
\[|w_1(\lambda(k))\ldots w_n(\lambda(k))u_n(k)-1|<\frac{1}{2}\quad\text{and}\quad |w_1(\mu(k))\ldots w_n(\mu(k))u_n(k)-1|<\frac{1}{2}.\]
Therefore,
\begin{align*}
	 \|(B_{w(\lambda(1))} \times & \cdots\times B_{w(\lambda(d))})^nu-(B_{w(\mu(1))}\times\cdots\times B_{w(\mu(d))})^nu\|  \\
	& \geq |w_1(\lambda(k))\ldots w_n(\lambda(k))u_n(k)-w_1(\mu(k))\ldots w_n(\mu(k))u_n(k)|\\
	&=|w_1(\lambda(k))\ldots w_n(\lambda(k))-w_1(\mu(k))\ldots w_n(\mu(k))|u_n(k)|\\
	& = \bigg|\frac{w_1(\lambda(k))\ldots w_n(\lambda(k))}{w_1(\mu(k))\ldots w_n(\mu(k))}-1\bigg||w_1(\mu(k))\ldots w_n(\mu(k))u_n(k)| \\
	&\geq \frac{1}{2}\Bigg|\bigg|\frac{w_1(\lambda(k))\ldots w_n(\lambda(k))}{w_1(\mu(k))\ldots w_n(\mu(k))}\bigg|-1\Bigg|\\
	&= \frac{1}{2} \big|\exp\big(f_n(\lambda(k))-f_n(\mu(k))\big)-1\big| \\
	&\geq \frac{1}{4}|f_n(\lambda(k))-f_n(\mu(k))|\\
	& \geq \frac{1}{4}cF(n)|\lambda(k)-\mu(k)|.
\end{align*}
Since this is true for each coordinate $1\leq k\leq d$, we obtain
\[\|(B_{w(\lambda(1))} \times \cdots\times B_{w(\lambda(d))})^nu-(B_{w(\mu(1))}\times\cdots\times B_{w(\mu(d))})^nu\| \geq  4^{-1}cF(n)\|\lambda-\mu\|. \]
Because our hypothesis implies that  \[\sum_{n\in\N}\phi\bigg(\frac{2\delta}{\psi(n)}\bigg) = \sum_{n\in\N}\frac{4^s}{c^sF(n)^s} <+\infty,\]
it follows from \cite[Theorem 3.1]{BCMparam} that the $s$-dimensional Hausdorff outer measure of $\Lambda$ is 0. 
\end{proof}

\begin{remark}
    Notice that one could have the same conclusion for any gauge function $\phi$ for which there is a function $\xi$ such that $\phi(tx)=\xi(t)\phi(x)$ for all $t,x>0$. 
\end{remark}

Applying the previous theorem for $d=1$ and $\phi(x)=x$, we conclude that $\sum 1/F(n)=+\infty$ is not only sufficient but also necessary for hypercyclicity, at least for shifts. Moreover, the result provides a necessary condition in fractional dimensions, although we know its sufficiency only in particular classes of backward shifts and fractal sets. This refers to Question \ref{Q:sumF:fini}.

\subsection{Common \texorpdfstring{upper frequent hypercyclicity }{UFHC} and dimension of the parameter set} \label{subsec:common:ufhc:neg}

We present some technical lemmas that will lead to the main result of the section. These results are  generalized versions of results from \cite{mestirithesis}, where we replaced $n^\beta$ by a general function $F$, the sets of parameters that we consider live in $\R^d$ for any $d\geq 1$ and we consider a general Borel measure in place of the Lebesgue measure.

\begin{notation} \label{notation:F}
Let us define the family $\Psi$ consisting of all non-decreasing functions $\psi:\R_+\to\R_+$ for which there exists some constant $\eta>0$ such that:
	\begin{enumerate}[(I)]
		\item \label{modif_teo:non-ex:ii} for any $t>0$, $\big(\frac{\psi(n)}{\psi(tn)}\big)_{n\geq 1}$ is bounded;
		\item \label{modif_teo:non-ex:iii} $\lim_{n\to+\infty}\psi(n)=+\infty$;
		\item \label{modif_teo:non-ex:v} there is $t_0>0$ such that $\psi\in \mathcal{C}^1(t_0,+\infty)$ and $\frac{\psi(t)}{\psi'(t)}\geq (1+\eta)t$ for all $t>t_0$.
        \end{enumerate}
We also define the set $\Psi^*$ composed of the functions $\psi\in \Psi$ for which there exists $\gamma\in\N$ satisfying the additional condition
\begin{enumerate}[(I)] \setcounter{enumi}{3}
    \item \label{modif_teo:non-ex:best} $\sum_{n\geq 1}\frac{1}{\psi(\gamma^n)}<+\infty$.
    \end{enumerate}
Of course, we have necessarily $\gamma>1$.
\end{notation}

\begin{remark}\label{Rem:FInfSeries}
     It is not difficult to check that the assumptions \ref{modif_teo:non-ex:iii} and \ref{modif_teo:non-ex:v} imply that $\sum 1/\psi(n)=+\infty$. Indeed, thanks to condition \ref{modif_teo:non-ex:iii}, let us choose $t_1\geq t_0$ such that, for all $t\geq t_1$, $\psi(t)>0$. Now integrating over $(t_1,n)$ for positive integers $n>t_1$ the inequality $\psi'(t)/\psi(t)\leq 1/((1+\eta)t) $ given by condition \ref{modif_teo:non-ex:v}, we obtain that there exists a positive constant $C$  that only depends on $\psi$, $t_1$ and $\eta$ such that
        \[\forall \, n\gg1,\ \ \frac{C}{n^{1/(1+\eta)}}\leq\frac{1}{\psi(n)}.\]
        This gives the result.
    \end{remark}
    
\begin{remark} \label{remark:t0:big}
    As the reader can notice in the following, we do not necessarily need the functions $\psi\in\Psi$ to be defined on the whole $\R_+$, but rather on some interval of the form $(t_0,+\infty)$. This simple observation will be specially useful in Example \ref{example:prod:comp:log:ufhc}.
\end{remark}

\begin{remark} \label{remark:neg:monia}
    Notice that the affirmation ``there exists $\gamma>0$ such that $\sum_{n\geq 1}\frac{1}{\psi(\gamma^n)}<+\infty$'' (condition \ref{modif_teo:non-ex:best} in the definition of $\Psi^*$) is precisely the opposite of the condition \ref{monia:iii}\ref{cond:gamma:monia} of Theorem \ref{monia:thm}. Since the family $\Psi^*$ will be used to provide non-existence results of common upper frequently hypercyclic vectors, it follows that condition \ref{modif_teo:non-ex:best} is ``optimal''.
\end{remark}

\begin{lemma}\label{Lem:simpli_famille_psi}
	Let $(\alpha_n)_{n\in\N}$ be a decreasing sequence of positive numbers such that $\sum_{n\geq1}\alpha_n<+\infty$. Then, there exists an increasing sequence $(\beta_n)_{n\in\N}$ of positive numbers tending to $+\infty$ such that:
	\begin{enumerate}
		\item $\sum_{n\in\N}\beta_n \alpha_n<+\infty$; \label{Lem:series}
		\item $(\beta_n\alpha_n)_{n\in\N}$ decreases; 
        \label{Lem:increases} 
        \item $\beta_n\alpha_n<1$ for all $n\in\N$. \label{Lem:petit}
	\end{enumerate}
\end{lemma}
\begin{proof}
	We begin by constructing an increasing sequence $(\gamma_n)_{n\in\N}$ of positive numbers tending to $+\infty$ satisfying \eqref{Lem:series}.
	By the convergence of the series, its tail
    $\sum_{n\geq N}\alpha_n$ tends to zero as $N$ tends to infinity, hence we can construct an increasing sequence $(N_k)_{k\in\N}$ with $N_1=1$ such that for every $k\in\N$,
	\[\sum_{n\geq N_k}\alpha_n\leq 2^{-2k}.\]
	We define $\gamma_n=2^k$ when $n\in[N_k,N_{k+1}-1)$.
	Then, 
	\begin{align*}
		\sum_{n\geq 1}\gamma_n\alpha_n=\sum_{k\geq1}\sum_{n=N_k}^{N_{k+1}-1}\gamma_n\alpha_n
		=\sum_{k\geq1}2^k\sum_{n=N_k}^{N_{k+1}-1}\alpha_n
		&\leq \sum_{k\geq1}2^k\sum_{n\geq N_k}\alpha_n\leq \sum_{k\geq1}2^k2^{-2k}=1.
	\end{align*}
	In particular, we deduce $\gamma_1\alpha_1<1$. Now let us define another sequence $(\beta_n)_{n\in\N}$ with $\beta_1:=\gamma_1$ and, for every $n\geq 2$, $\beta_n:=\min(\beta_{n-1}\frac{\alpha_{n-1}}{\alpha_n}, \gamma_n)$. Then \eqref{Lem:increases} and \eqref{Lem:petit} are satisfied by definition of $\beta_n$. Moreover, \eqref{Lem:series} is also quite clear since $\beta_n\leq \gamma_n$ for every $n\in\N$ and $\sum_{n\geq 1}\gamma_n\alpha_n<+\infty$.
    The convergence of $(\beta_n)_{n\in\N}$ to infinity is the last property left to prove. Let us first prove that $(\beta_n)_{n\in\N}$ is increasing. For every $n\in\N$, since $(\alpha_n)_{n\in\N}$ is decreasing and $(\gamma_n)_{n\in\N}$ is increasing,
\[\beta_n=\min\Big(\beta_{n-1}\frac{\alpha_{n-1}}{\alpha_n},\gamma_n\Big)\geq \min(\beta_{n-1},\gamma_{n-1})\geq \beta_{n-1}.\]
    Hence, $(\beta_n)_{n\in\N}$ increases. Recall that $\beta_n=\min(\beta_{n-1}\frac{\alpha_{n-1}}{\alpha_n}, \gamma_n)$ and remark that $\beta_n$ cannot eventually take the first value of the $\min$. Indeed, in that case, $(\alpha_n\beta_{n})_{n\in\N}$ would be eventually constant and the series $\sum_{n\in\N}\alpha_n\beta_{n}$ would not be convergent. So, the $\beta_n$'s have to be equal to the $\gamma_n$'s for infinitely many $n\in\N$. However, $(\beta_n)_{n\in\N}$ being increasing and $(\gamma_n)_{n\in\N}$ tending to infinity then $(\beta_n)_{n\in\N}$ must also tend to infinity.
\end{proof}

\begin{corollary} \label{corol:modif_teo:non-ex}
    Let $\psi\in\Psi^*$. Then, there exists a positive function $h:\N\to\R_{+}^{*}$ such that:
    \begin{enumerate}[(I)]
         \setcounter{enumi}{4}
		\item \label{modif_teo:non-ex:iv} $\lim_{n\to+\infty}h(n)=0$;
		\item \label{modif_teo:non-ex:vi} $\sum_{n\geq 1}\frac{1}{h(n)\psi(\gamma^n)}<+\infty$;
		\item \label{modif_teo:non-ex:vii} the sequence $(h(n)\psi(\gamma^n))_{n\geq 1}\subset (1,+\infty)$ is increasing. 
	\end{enumerate}
\end{corollary}

\begin{proof}
    It suffices to apply Lemma~\ref{Lem:simpli_famille_psi} with $\alpha_n=\frac{1}{\psi(\gamma^n)}$ and  define $h(n):=\frac{1}{\beta_n}$.
\end{proof}

Along this section, when we pick a function $\psi\in\Psi^*$, we implicitly apply the previous corollary and consider that $\psi$ comes with the function $h$ given by the corollary and satisfying \ref{modif_teo:non-ex:iv}, \ref{modif_teo:non-ex:vi} and \ref{modif_teo:non-ex:vii}.

\begin{lemma}
	\label{modif_lemma:lower:bound:N}
	Let $k,N\in\N$, $\beta>1$ and $\psi\in\Psi$. Let also $A>0$ satisfying $A\geq \frac{1}{\psi(N)}$. Then
	\[N_k=\min\bigg\{M\in\N : kA\leq \sum_{n=N+1}^{M}\frac{1}{\psi(n)}< (k+1)A\bigg\}\]
	exists and satisfies \[\frac{N_k}{\psi(N_k)}\geq \Big(1-\frac{1}{1+\eta}\Big)kA+\frac{N}{\psi(N)}.\]
\end{lemma}
\begin{proof}
	The conditions on $A$ and $\psi$ and also Remark~\ref{Rem:FInfSeries} ensure the existence of $N_k$.
	By definition, we remark that 
	\[kA\leq \sum_{n=N+1}^{N_k} \frac{1}{\psi(n)}.\]
	On the other hand, we look for an upper bound of the preceding sum. We have
	\[\sum_{n=N+1}^{N_k} \frac{1}{\psi(n)}\leq \int_{N}^{N_k}\frac{1}{\psi(t)}dt=\Big(\frac{N_k}{\psi(N_k)}-\frac{N}{\psi(N)}\Big)+\int_{N}^{N_k}t\frac{\psi'(t)}{\psi^2(t)}dt.\]
	By the hypothesis on $\psi$, we ensure that $\int_{N}^{N_k}t\frac{\psi'(t)}{\psi^2(t)}dt\leq \int_{N}^{N_k}\frac{1}{(1+\eta)\psi(t)}dt$, hence 
	\[\int_{N}^{N_k}\frac{1}{\psi(t)}dt\leq\Big(\frac{N_k}{\psi(N_k)}-\frac{N}{\psi(N)}\Big)+\int_{N}^{N_k}\frac{1}{(1+\eta)\psi(t)}dt\]
	which gives the estimation
	\[\int_{N}^{N_k}\frac{1}{\psi(t)}dt\leq\Big(1+\frac{1}{\eta}\Big)\Big(\frac{N_k}{\psi(N_k)}-\frac{N}{\psi(N)}\Big),\]
	which yields the upper bound
	\[\sum_{n=N+1}^{N_k} \frac{1}{\psi(n)}\leq\Big(1+\frac{1}{\eta}\Big)\Big(\frac{N_k}{\psi(N_k)}-\frac{N}{\psi(N)}\Big).\]
	Finally, gathering these inequalities, we obtain
	\[kA\leq \Big(1+\frac{1}{\eta}\Big)\Big(\frac{N_k}{\psi(N_k)}-\frac{N}{\psi(N)}\Big)\]
	which leads to 
	\[\frac{N_k}{\psi(N_k)}\geq \Big(1-\frac{1}{1+\eta}\Big)kA +\frac{N}{\psi(N)}.\]
\end{proof}

In the sequel, we will need to impose some regularity conditions on our family $(T_\lambda)_{\lambda\in\Lambda}$. In our case, we do not need to impose the continuity of the family, but it suffices to demand that 
$\lambda\mapsto T_\lambda u$ is continuous for all $u\in X$. In particular, in the case of weighted shifts, we will not need any monotonicity assumptions on the weights.

\begin{lemma}
\label{modif_lemma:inversion:measure} Let $\Lambda$ be a Borel subset of $\R^d$, $d\geq1$, and $\Mes$ be a Borel measure. Let $N,M,k\in\N$ with $N<M$. Let $X$ be an $F$-space, $V\subset X$ be open and non-empty and let $u\in X$. Consider a family $(T_\lambda)_{\lambda\in\Lambda}\subset \L(X)$ such that $\lambda\mapsto T_\lambda v$ is continuous for all $v\in X$ and define \[J_k=\{\lambda\in\Lambda : \# (N_\lambda(u,V)\cap [N+1, M])\geq k\}.\]
Then 
\[\Mes(J_k)\leq \frac{1}{k}\sum_{n=N+1}^M \Mes(\{\lambda\in J_k : n\in N_\lambda(u,V)\}).\]
\end{lemma}
\begin{proof}
By the definition of $J_{k}$ we observe that, for any $1\leq l\leq k$ and any $\lambda\in J_k$, 
\[
    \exists \, n\in N_{\lambda}(u,V)\cap[N+1,M] \quad \text{and} \quad \#(N_{\lambda}(u,V)\cap[N+1,n])=l.
\]
We then denote, for $1\leq l\leq k$ and for $n\in[N+1,M],$
\[
    U_{l,n}:=\{\lambda\in J_{k} \mid n\in N_{\lambda}(u,V) \text{ and } \#(N_{\lambda}(u,V)\cap[N+1,n])=l\}.
\]
With this notation the previous assertion becomes
\[
    \forall \, 1\leq l\leq k, \quad J_{k} \subset \bigcup_{n=N+1}^{M} U_{l,n}.
\]
By sub-additivity of the measure, this implies that
\[
    \forall \, 1\leq l\leq k, \quad \Mes(J_{k}) \leq \sum_{n=N+1}^{M} \Mes(U_{l,n}).
\]
It then follows that
\begin{equation} \label{eq3.16_Monia}
    k \Mes(J_{k}) = \sum_{l=1}^{k}\Mes(J_{k}) \leq \sum_{l=1}^{k}\sum_{n=N+1}^{M}\Mes(U_{l,n}) = \sum_{n=N+1}^{M} \sum_{l=1}^{k}\Mes(U_{l,n}).
\end{equation}

Moreover, for any $n\in[N+1,M]$ the sets $U_{l,n}$ with $1\leq l\leq k$, are pairwise disjoint. Therefore we obtain
\[
    \forall \, n\in[N+1,M], \quad \Mes\bigg(\bigcup_{l=1}^{k} U_{l,n}\bigg) = \sum_{l=1}^{k} \Mes(U_{l,n}).
\]
Now using \eqref{eq3.16_Monia}, this entails that
\[
    k \Mes(J_{k}) \leq \sum_{n=N+1}^{M} \Mes\bigg(\bigcup_{l=1}^kU_{l,n}\bigg).
\]
On the other hand we have, by definition, that
\[
    \forall \, n\in[N+1,M], \quad \bigcup_{l=1}^k U_{l,n} \subset \{\lambda \in J_{k} \mid n \in N_{\lambda}(u, V)\}.
\]
From the monotonicity of the measure, we then deduce that
\[
    k \Mes(J_{k}) \leq \sum_{n=N+1}^{M} \Mes(\{\lambda\in J_{k} \mid n\in N_{\lambda}(u,V)\}),
\]
which proves the claim.
\end{proof}

\begin{lemma}\label{modif_lemma:bornes} Let $\Lambda$ be a Borel subset of $\R^d$, $d\geq1$, and $\Mes$ be a Borel measure with $0<\Mes(\Lambda)<+\infty$.
	Let $X$ be a separable $F$-space, $V\subset X$ open and non-empty and $u\in X$. Let $(T_\lambda)_{\lambda\in\Lambda}\subset \L(X)$ be a family such that $\lambda\mapsto T_\lambda v$ is continuous for all $v\in X$. Suppose that there exist $c>0$ and $\psi\in \Psi$ such that 
    \begin{equation}\label{eq:lemme:technique}
	    \forall \, n\gg 1, \, \Mes(\{\lambda\in\Lambda : n\in N_{\lambda}(u,V)\})\leq \frac{c}{\psi(n)}.
	\end{equation}
    Let $N_0\in\N$, $(\alpha_l)_{l\geq 1}\subset (1,+\infty)$ and $ (k_l)_{l\geq 1}\subseteq\N$
    satisfying 
	\begin{enumerate}[(1)]
		\item \label{modif_lemma:cond:alpha:k}
        $\sum_{n\geq 1}\frac{1}{\alpha_n}<+\infty$, $s:=\sum_{n\geq 1}\frac{1}{k_n}<+\infty$,
		\item \label{modif_lemma:cond:N0} $\psi(N_0)\geq  \frac{2c}{\Mes(\Lambda)}\max\big(1,\frac{1}{C_{\alpha}},\frac{s}{C_{\alpha}}\big)$, where $C_\alpha:=\prod_{n\geq1}\big(1-\frac{1}{\alpha_n}\big)$.
	\end{enumerate}
	Then there exist some increasing sequence $(N_l)_{l\geq 1}$ of positive integers and some non-increasing sequence $(I_l)_{l\geq 1}$ of subsets of $\Lambda$ such that, for all $l\geq 1$,
	\begin{enumerate}[(i)]
		\item \label{modif_lemma:cond:i} $I_0=\Lambda$ and $I_l=\{\lambda\in I_{l-1} : \# (N_\lambda(u,V))\cap [N_{l-1}+1,N_l]<\alpha_lk_l\}$,
		\item \label{modif_lemma:cond:ii} $\frac{N_l}{\psi(N_l)}\geq \frac{\eta}{c(1+\eta)}k_l\Mes(I_{l-1})+\frac{N_{l-1}}{\psi(N_{l-1})}$,
		\item \label{modif_lemma:cond:iii} $\Mes(I_l)\geq \Mes(\Lambda)\prod_{i=1}^l\big(1-\frac{1}{\alpha_i}\big) - \sum_{i=1}^l\frac{1}{\alpha_ik_i}\frac{c}{\psi(N_i)}$,
		\item \label{modif_lemma:cond:iv} $\Mes(I_l)\geq \frac{\Mes(\Lambda)}{2}C_{\alpha}$.
	\end{enumerate}
\end{lemma}
\begin{proof}
	First notice that the convergence of the series $\sum_{n\geq1} \frac{1}{\alpha_n}$ ensures that $\big(\prod_{i=1}^l (1-1/\alpha_i)\big)_{l\geq 1}$ is a decreasing sequence with $\prod_{i=1}^l (1-1/\alpha_i) \geq C_{\alpha}$.
	
	Up to choosing a bigger $N_0$, we can assume that inequality \eqref{eq:lemme:technique} holds true for all $n\geq N_0$. The construction is done by induction on $l$. For $l=1$, let $k=k_1$, $N=N_0$ and $A=\frac{1}{c}\Mes(I_0)=\frac{1}{c}\Mes(\Lambda)$. From condition \ref{modif_lemma:cond:N0} it follows that $A\geq \frac{1}{\psi(N_0)}$ so that we can apply Lemma \ref{modif_lemma:lower:bound:N} and get $N_1>N_0$ such that 
	\begin{enumerate}[(a)]
		\item \label{modif_lemma:cond:I} $N_1= \min\big\{M\in\N : k_1\Mes(\Lambda)\leq \sum_{n=N_0+1}^{M}\frac{c}{\psi(n)}< (k_1+1)\Mes(\Lambda)\big\}$,
		\item \label{modif_lemma:cond:II} $\frac{N_1}{\psi(N_1)}\geq \frac{\eta}{c(1+\eta)}k_1\Mes(\Lambda)+\frac{N_0}{\psi(N_0)}$.
	\end{enumerate}
	Condition \ref{modif_lemma:cond:ii} is exactly \ref{modif_lemma:cond:II}. Defining
	\[I_1:=\{\lambda\in\Lambda : \#(N_\lambda(u,V)\cap[N_0+1,N_1])<\alpha_1 k_1\},\]
	condition \ref{modif_lemma:cond:i} also follows. Now, applying Lemma \ref{modif_lemma:inversion:measure} with $N=N_0$, $M=N_1$, $k=\alpha_1k_1$, using the hypothesis and noticing that $\Lambda\backslash I_1\subset J_k$, we get
	\[\Mes(\Lambda\backslash I_1) \leq  \frac{1}{\alpha_1k_1}\sum_{n=N_0+1}^{N_1}\Mes(\{\lambda\in \Lambda : n\in N_\lambda(u,V)\})\leq \frac{1}{\alpha_1k_1}\sum_{n=N_0+1}^{N_1}\frac{c}{\psi(n)}.\]
	From this and from \ref{modif_lemma:cond:I} we find 
	\[\Mes(\Lambda\backslash I_1) \leq \frac{1}{\alpha_1k_1}\sum_{n=N_0+1}^{N_1-1}\frac{c}{\psi(n)} + \frac{1}{\alpha_1k_1}\frac{c}{\psi(N_1)}\leq \frac{1}{\alpha_1}\Mes(\Lambda)+\frac{1}{\alpha_1 k_1}\frac{c}{\psi(N_1)}.
	\]
	Therefore,
	\[\Mes(I_1)\geq \Mes(\Lambda)\Big(1-\frac{1}{\alpha_1}\Big)-\frac{1}{\alpha_1 k_1}\frac{c}{\psi(N_1)},\]
	which corresponds to condition \ref{modif_lemma:cond:iii}. Finally, since $\alpha_1> 1$, $k_1\geq 1$, $N_1>N_0$ and $(1-1/\alpha_1)\geq C_{\alpha}$, it follows from \ref{modif_lemma:cond:N0} that 
	\[\frac{c}{\alpha_1 k_1 \psi(N_1)}\leq \frac{\Mes(\Lambda)C_{\alpha}}{2} \implies \Mes(I_1)\geq \frac{\Mes(\Lambda)}{2}C_{\alpha},\]
	which is \ref{modif_lemma:cond:iv}. This completes the construction for $l=1$.
	
	Let $l\geq 1$ and suppose that we have defined $N_1<N_2<\cdots<N_l$ positive integers and $I_1\supset I_2\supset\ldots\supset I_l$ subsets of $\Lambda$ satisfying conditions \ref{modif_lemma:cond:i}-\ref{modif_lemma:cond:iv}. Let $k=k_{l+1}$, $N=N_l$ and $A=\frac{1}{c}\Mes(I_l)$. From \ref{modif_lemma:cond:N0} and since $N_l>N_0$, it follows that $A\geq \frac{1}{\psi(N_l)}$. Therefore, we can apply Lemma \ref{modif_lemma:lower:bound:N} and get $N_{l+1}>N_l$ such that
	\begin{enumerate}[(a)]
		\setcounter{enumi}{2}
		\item \label{modif_lemma:cond:III} $N_{l+1}=\min\bigg\{M\in\N : k_{l+1}\Mes(I_l)\leq \sum_{n=N_l+1}^{M}\frac{c}{\psi(n)}< (k_{l+1}+1)\Mes(I_l)\bigg\}$
		\item \label{modif_lemma:cond:IV} $\frac{N_{l+1}}{\psi(N_{l+1})}\geq \frac{\eta}{c(1+\eta)}k_{l+1}\Mes(I_l)+\frac{N_l}{\psi(N_l)}$.
	\end{enumerate}
	Condition \ref{modif_lemma:cond:IV} is exactly \ref{modif_lemma:cond:ii}. Defining
	\[I_{l+1}:=\{\lambda\in I_l : \#(N_\lambda(u,V)\cap[N_l+1,N_{l+1}])<\alpha_{l+1}k_{l+1}\},\]
	condition \ref{modif_lemma:cond:i} also follows. Now, applying Lemma \ref{modif_lemma:inversion:measure} with $N=N_l$, $M=N_{l+1}$, $k=\alpha_{l+1}k_{l+1}$, using the hypothesis and noticing that $I_l\backslash I_{l+1}\subset J_k$, we get
	\[\Mes(I_l\backslash I_{l+1})\leq \frac{1}{\alpha_{l+1}k_{l+1}}\sum_{n=N_l+1}^{N_{l+1}}\Mes(\{\lambda\in \Lambda:n\in N_\lambda(u,V)\})\leq \frac{1}{\alpha_{l+1}k_{l+1}}\sum_{n=N_l+1}^{N_{l+1}} \frac{c}{\psi(n)}.\]
	From this and from \ref{modif_lemma:cond:III} we find
	\begin{align*}
		\Mes(I_l\backslash I_{l+1}) 
		&\leq \frac{1}{\alpha_{l+1}k_{l+1}}\sum_{n=N_l+1}^{N_{l+1}-1}\frac{c}{\psi(n)}+
		\frac{1}{\alpha_{l+1}k_{l+1}}\frac{c}{\psi(N_{l+1})} \\
		&\leq \frac{1}{\alpha_{l+1}}\Mes(I_l) + \frac{1}{\alpha_{l+1}k_{l+1}} \frac{c}{\psi(N_{l+1})},\end{align*}
        which implies 
        \[
        \Mes(I_{l+1})
		\geq \Mes(I_l)\Big(1-\frac{1}{\alpha_{l+1}}\Big) - \frac{1}{\alpha_{l+1}k_{l+1}} \frac{c}{\psi(N_{l+1})}.
	\]
	By using condition \ref{modif_lemma:cond:iii} in the induction hypothesis we get
	\begin{align*}
		\Mes(I_{l+1}) 
		&\geq \Mes(\Lambda)\prod_{i=1}^{l+1}\Big(1-\frac{1}{\alpha_i}\Big)-\sum_{i=1}^{l+1}\frac{c}{\alpha_ik_i\psi(N_i)}+\frac{1}{\alpha_{l+1}}\sum_{i=1}^l\frac{c}{\alpha_ik_i\psi(N_i)}\\
		&\geq \Mes(\Lambda)\prod_{i=1}^{l+1}\Big(1-\frac{1}{\alpha_i}\Big)-\sum_{i=1}^{l+1}\frac{c}{\alpha_ik_i\psi(N_i)},
	\end{align*}
	thus implying \ref{modif_lemma:cond:iii}. Furthermore, from \ref{modif_lemma:cond:N0} and \ref{modif_lemma:cond:alpha:k}, we have
	\[
	\sum_{i=1}^{l+1} \frac{c}{\alpha_i k_i \psi(N_i)}\leq \frac{c}{\psi(N_0)}\sum_{i=1}^{l+1} \frac{1}{k_i} \leq \frac{sc}{\psi(N_0)} \leq \frac{\Mes(\Lambda)}{2}C_{\alpha}.
	\]
	Therefore,
	\[
	\Mes(I_{l+1})\geq \Mes(\Lambda)\prod_{i=1}^{l+1}\Big(1-\frac{1}{\alpha_i}\Big) - \sum_{i=1}^{l+1} \frac{c}{\alpha_i k_i \psi(N_i)} \geq \Mes(\Lambda)C_{\alpha}-\frac{\Mes(\Lambda)}{2}C_{\alpha},
	\]
	ultimately implying \ref{modif_lemma:cond:iv}. This completes the proof.
\end{proof}

The next result provides sufficient conditions on the non-existence of common upper frequent hypercyclic vectors for a general family of operators $(T_\lambda)_{\lambda\in\Lambda}$ acting on the same $F$-space $X$.

\begin{theorem} \label{modif_teo:non-ex} Let $\Lambda$ be a Borel subset of $\R^d$, $d\geq1$, and $\Mes$ be a Borel measure  with $0<\Mes(\Lambda)<+\infty$.
	Let $X$ be a separable $F$-space and consider a family $(T_\lambda)_{\lambda\in\Lambda}\subset \L(X)$ such that $\lambda\mapsto T_\lambda u$ is continuous for all $u\in X$  and $\psi\in\Psi^*$.
	Assume that, for all $u\in X$, there are $V\subset X$ a non-empty open set and $c>0$ such that
	\begin{equation}\label{modif_teo:non-ex:eq}
		\forall \, n\gg 1, \, \Mes(\{\lambda\in\Lambda:n\in N_{\lambda}(u,V)\})\leq \frac{c}{\psi(n)}.
	\end{equation}
	Then $(T_\lambda)_{\lambda\in \Lambda}\subset \L(X)$ has no common upper frequently hypercyclic vector.
\end{theorem}

\begin{proof}
	Let $u\in X$ arbitrary. We aim to find $\lambda_0\in\Lambda$ with $u\notin \mathcal{U}FHC(T_{\lambda_0})$. The hypothesis says that there is $V\subset X$ open and non-empty such that \eqref{modif_teo:non-ex:eq} holds for all $n$ sufficiently big. 
    We define $\alpha_l=h(l)\psi(\gamma^l)$ and $k_l=\gamma^{l}$ with $\gamma\in\mathbb{N}\setminus\{1\}$, where $h:\N\to \R_+^*$ is given by 
    Corollary \ref{corol:modif_teo:non-ex}. Notice that $\sum_{l\geq 1}\frac{1}{\alpha_l}$ converges from condition \ref{modif_teo:non-ex:vi}. We get that condition \ref{modif_lemma:cond:alpha:k} of Lemma \ref{modif_lemma:bornes} is satisfied. We then fix $N_0\geq t_0\in\N$ such that \ref{modif_lemma:cond:N0} of Lemma \ref{modif_lemma:bornes} is satisfied. Therefore we apply Lemma \ref{modif_lemma:bornes} in order to get an increasing sequence of positive integers $(N_l)_{l\geq 1}$ as well as a non-increasing sequence of subsets $(I_l)_{l\geq 1}$ of $\Lambda$ satisfying the conclusions \ref{modif_lemma:cond:i}-\ref{modif_lemma:cond:iv} of Lemma \ref{modif_lemma:bornes}. We define $I=\bigcap_{l\in\N}I_l$, what immediately implies that $\Mes(I_l)\to \Mes(I)$ as $l\to+\infty$. From \ref{modif_lemma:cond:iv} of Lemma \ref{modif_lemma:bornes} we conclude that $\Mes(I)$ is bounded away from zero. In particular, $I\neq \varnothing$. Let us fix $\lambda_0\in I$ and prove that $u\notin \U FHC(T_{\lambda_0})$. We shall first prove that 
	\begin{equation}\label{modif_lemma:eq:lim}
		\frac{\#(N_{\lambda_0}(u,V)\cap[0,N_{l+1}])}{N_l+1}\xrightarrow{l\to+\infty}0.
	\end{equation}
	Since $\lambda_0\in I_l$ for all $l\geq 1$, we know from the assertion \ref{modif_lemma:cond:i} of Lemma \ref{modif_lemma:bornes} that \[\#(N_{\lambda_0}(u,V)\cap[N_{l-1}+1,N_l])<\alpha_lk_l, \, \forall \, l\geq 1.\]
	Therefore,
	\[\#(N_{\lambda_0}(u,V)\cap[N_{0}+1,N_{l+1}])<\sum_{i=1}^{l+1}\alpha_ik_i, \, \forall \, l\geq 1.\]
	Then for $l\geq1$, we can use condition \ref{modif_teo:non-ex:vii} to get 
	\begin{align*}
		\sum_{i=1}^l\alpha_ik_i 
		= \sum_{i=1}^l h(i)\psi(\gamma^i)\gamma^{i}
		\leq h(l)\psi(\gamma^l)\sum_{i=1}^l\gamma^{i}
		= h(l)\psi(\gamma^l)\frac{\gamma^{l+1}-\gamma}{\gamma-1}.
	\end{align*}
	On the other hand, from conditions \ref{modif_lemma:cond:ii} and \ref{modif_lemma:cond:iv} of Lemma \ref{modif_lemma:bornes}, a simple induction gives
	\begin{align*}
		\frac{N_l}{\psi(N_l)}
		&\geq \frac{\eta}{c(1+\eta)}k_l\Mes(I_{l-1}) + \frac{N_{l-1}}{\psi(N_{l-1})}\\
		&\geq \frac{\eta}{c(1+\eta)}k_l\Mes(I_{l-1}) + \frac{\eta}{c(1+\eta)}k_{l-1}\Mes(I_{l-2}) + \frac{N_{l-2}}{\psi(N_{l-2})}\\
		&\, \, \, \vdots \\
		&\geq \frac{\eta}{c(1+\eta)}\sum_{i=1}^l k_i\Mes(I_{i-1}) \\
		&\geq \frac{\eta C_\alpha\Mes(\Lambda)}{2c(1+\eta)} \sum_{i=1}^lk_i\\
		&\geq \frac{\eta C_\alpha\Mes(\Lambda)}{2c(1+\eta)}\gamma^l.
	\end{align*}
	Since $\psi$ is non-decreasing and using \ref{modif_teo:non-ex:ii}, we get
	\begin{align*}
		N_l&\geq K\gamma^l\psi(N_l)\geq K\gamma^l\psi(K\gamma^l\psi(N_l))\geq K\gamma^l\psi(K\gamma^l)\gtrsim \gamma^{l+1}\psi(\gamma^{l+1}),
	\end{align*}
	where $K=\frac{\eta C_\alpha\Mes(\Lambda)}{2c(1+\eta)}$. Therefore,
	\begin{align*}
		\frac{\#(N_{\lambda_0(u,V)}\cap[0,N_{l+1}])}{N_l+1}
		&\leq \frac{N_0+1}{N_l+1}+\frac{1}{N_l+1}\sum_{i=1}^{l+1}\alpha_ik_i \\
		&\lesssim \frac{N_0+1}{N_l+1} 
		+ \frac{h(l+1)\psi(\gamma^{l+1})(\gamma^{l+2}-\gamma)}{(\gamma^{l+1}\psi(\gamma^{l+1})+1)(\gamma-1)}
		\xrightarrow{l\to+\infty}0,
	\end{align*}
	where the last limit holds because $h(l)\to 0$ as $l\to+\infty$, as established by Corollary \ref{corol:modif_teo:non-ex}. This proves \eqref{modif_lemma:eq:lim}. Let us now check that this implies $\overline{d} (N_{\lambda_0}(u,V))=0$. Let $\veps>0$. From \eqref{modif_lemma:eq:lim}, there exists $l_0\in\N$ such that \[l\geq l_0\implies \frac{\#(N_{\lambda_0}(u, V)\cap[0,N_{l+1}])}{N_l+1}<\veps.\]
	Let $N\geq N_{l_0}$. Since $(N_l)_{l\geq 1}$ is increasing, there is $l \geq 1$ such that $N_l\leq N\leq N_{l+1}$, what gives us
	\[\frac{\#(N_{\lambda_0}(u, V)\cap[0,N])}{N+1}\leq \frac{\#(N_{\lambda_0}(u, V)\cap[0,N_{l+1}])}{N_l+1}<\veps,\]
	and this implies $\overline{d}(N_{\lambda_0}(u,V))=0$ as we wanted. Therefore, $u\notin \U FHC(T_{\lambda_0})$ and the proof is complete.
\end{proof}

Our results apply to families of operators indexed by intricate sets of parameters, namely self-similar fractals. Thus, for the following application, we must recall some geometric measure theoretic notions. Given $d,N\geq 1$, we call \emph{iterated function system} (IFS) a set $\mathcal{S}=\{S_1,\ldots,S_N\}$ where each $S_i:\R^d\to \R^d, i=1,\dots,N$, is a \emph{similitude} (composition of rotation, contraction/dilation and translation) with contraction ratio $r_i>0$. We say that $\mathcal{S}$ satisfies the \emph{Open Set Condition} (OSC) if there exists $U\subset \R^d$ non-empty bounded and open such that $S_1(U),\ldots, S_N(U)$ are pairwise disjoint and
    \[\bigsqcup_{i=1}^{N} S_i(U)\subseteq U.\] Every IFS has an invariant set $\Lambda$ which is the only compact subset of $\R^d$ satisfying
    \[\Lambda = \bigcup_{i_1=1}^N\cdots\bigcup_{i_m=1}^N S_{i_1}\circ \cdots \circ S_{i_m}(\Lambda), \quad \forall \, m\geq 1.\]
    This set $\Lambda$ is often called a \emph{self-similar fractal} with $N$ similarities and contraction ratios $r_1,\dots, r_N> 0$, although the definition can include sets with no ``fractal'' structure at all. When $\mathcal{S}$ satisfies the OSC, $\Lambda$ has interesting special properties. For instance, the unique number $s\geq 0$ such that $\sum_{i=1}^Nr_i^s=1$ coincides with the Hausdorff dimension of $\Lambda$. If $\mathcal{S}$ has the OSC, we say that $\Lambda$ has the OSC.

    The following lemma will help in obtaining our next results.
    
\begin{lemma}\label{prop:connect:newnew}
	Given $d\geq 1$, let $\Lambda\subset\R^d$ be a self-similar fractal with Hausdorff dimension $s>0$ and satisfying the OSC. Consider a family $(T_\lambda)_{\lambda\in\Lambda}\subset \L(X)$ such that $\lambda\mapsto T_\lambda u$ is continuous for all $u\in X$ and suppose that there is a non-decreasing function $F:\N\to(0,+\infty)$ such that $\lim_{n\to\infty}F(n)=+\infty$ and there is $v\in X\backslash\{0\}$ and $0<\delta<\|v\|$ such that,
	for every $u\in X\setminus\{0\}$, there exists a constant $C_{u,V}>0$ such that, for every $n\gg1$ and every $\lambda,\mu \in \Lambda$ with $n\in N_{\lambda}(u,V)$,
	\begin{equation}\label{eq:F:lips:inv2}
		\|\lambda-\mu\|_{\infty} \ge \frac{C_{u,V}}{F(n)}\implies n\notin N_{\mu}(u,V),
	\end{equation}
	where $V=B(v,\delta)$. Then, for all $u\in X$, there is $c>0$ such that 
	\begin{equation}
		\label{prop:key:eq:02}
		\forall \, n\gg 1, \mathcal{H}^s(\{\lambda\in \Lambda : n\in N_\lambda(u,V)\})\leq \frac{c}{F(n)^s}.
	\end{equation}
\end{lemma}
\begin{proof}
	It follows from \cite[Theorem 4.14]{Mattila} that $0<\mathcal{H}^s(\Lambda)<+\infty$ and there exists $c_0\in (0,+\infty)$ such that, for every $\lambda\in\Lambda$ and every $r\in (0,1]$, 
	\begin{equation}\label{eqConsOSC2}
		\mathcal{H}^{s}(\Lambda\cap B(\lambda,r))\leq c_0r^s.
	\end{equation}
	Let $F$, $v\in X\backslash\{0\}$ and $V=B(v, \delta)$ as in the statement.
	Let $u\in X$ be arbitrary. If $u=0$, then \eqref{prop:key:eq:02} is trivial because $T_{\lambda}^{n}(u)=0$ for every $n\in\N$. Suppose $u\neq 0$ and let $n\gg 1$ such that $F(n)\geq 1$. Define 
	\[K = \{\lambda\in \Lambda : n\in N_\lambda(u,V)\}.\]
    If $K=\emptyset$, inequality \eqref{prop:key:eq:02} is trivially satisfied. Assume now that $K\neq\emptyset$.
	Let $\lambda, \mu \in K$. By definition of $K$, we have $n \in N_{\lambda}(u,V)$ and $n \in N_{\mu}(u,V)$.
		By the assumption of the proposition, this directly implies that $\|\lambda-\mu\|_{\infty} < \frac{C_{u,V}}{F(n)}$. Therefore, $K \subset B(\lambda_0, C_{u,V}/F(n))$ for any fixed $\lambda_0 \in K$.	
	Thus, from \eqref{eqConsOSC2} it follows that
	\[
	\mathcal{H}^{s}(\{\lambda\in\Lambda:n\in N_{\lambda}(u,V)\})\leq \frac{c}{F(n)^{s}}
	\]
	with $c=C_{u,V}^s c_0$, as we wanted.
\end{proof}
The following proposition connects the reverse Lipschitz condition \eqref{cond:comp-3} to our non-existence results for the case of families $(T_\lambda)_{\lambda\in \Lambda}$ acting on a separable $F$-space $X$ and indexed by a self-similar fractal $\Lambda\subset \R^d$, $d\geq1$.

\begin{proposition}
    \label{prop:connect:new2}
    Given $d\geq 1$, let $\Lambda\subset\R^d$ be a self-similar fractal with Hausdorff dimension $s>0$ and satisfying the OSC. Consider a family $(T_\lambda)_{\lambda\in\Lambda}\subset \L(X)$ such that $\lambda\mapsto T_\lambda u$ is continuous for all $u\in X$ and suppose that there is a non-decreasing function $F:\N\to(0,+\infty)$ such that $\lim_{n\to\infty}F(n)=+\infty$ and there is $v\in X\backslash\{0\}$ and $0<\delta<\|v\|$ such that,  
    for every $u\in X\setminus\{0\}$, there exists $c_{u,v}>0$ 
    satisfying, for every $n\gg1$ and every $\lambda,\mu\in\Lambda$ with $n\in N_{\lambda}(u,V)$,
    \begin{equation}\label{eq:F:lips:inv}
        c_{u,v}F(n)\|\lambda-\mu\|_\infty\leq \|T_{\lambda}^{n}u-T_{\mu}^{n}u\|,
    \end{equation}
    where $V=B(v,\delta)$. Then, for all $u\in X$, there is 
    $c>0$ such that 
    \begin{equation}
    \label{prop:key:eq:0}
    \forall \, n\gg 1, \mathcal{H}^s(\{\lambda\in \Lambda : n\in N_\lambda(u,V)\})\leq \frac{c}{F(n)^s}.
    \end{equation}
\end{proposition}

\begin{proof}
    The hypotheses provide $c_{u,v}>0$ satisfying inequality \eqref{eq:F:lips:inv}. Choosing $C_{u,V}=\frac{2\delta}{c_{u,v}}>0$ in Lemma~\ref{prop:connect:newnew}, it remains only to verify the implication \eqref{eq:F:lips:inv2}. So let $n\gg1$, $\lambda,\mu\in \Lambda$ with $n\in N_\lambda(u,V)$ and suppose that \begin{equation}\label{eq:F:lips:inv:proof}\|\lambda-\mu\|_\infty\geq \frac{C_{u,V}}{F(n)}=\frac{2\delta}{c_{u,v}F(n)}.
    \end{equation}
    Since $n\in N_\lambda(u,V)$ we know that $\|T_\lambda^n u - v\|<\delta$. Therefore, from \eqref{eq:F:lips:inv} and \eqref{eq:F:lips:inv:proof}, we get
\begin{align*}
    \|T_\mu^nu-v\|
        &\geq \|T_\mu^nu-T_\lambda^nu\|- \|T_\lambda^nu-v\| \\
        &\geq c_{u,v}F(n)\|\lambda-\mu\|_\infty - \delta\\
        &> 2\delta - \delta = \delta,
\end{align*}
which implies that $n\notin N_\mu(u,V)$ and shows the implication \eqref{eq:F:lips:inv2} as we wanted.
\end{proof}

\begin{remark}\label{remarque:d=1}
Notice that, when $d=1$, the same proofs show that one may weaken the hypotheses in Lemma~\ref{prop:connect:newnew} and Proposition~\ref{prop:connect:new2} by asking that, for every $n\gg1$ and every $\lambda<\mu$ in $\Lambda$ with $n\in N_{\lambda}(u,V)$, inequality \eqref{eq:F:lips:inv} holds.
\end{remark}

With Proposition~\ref{prop:connect:new2} at hand, we can use Theorem \ref{modif_teo:non-ex} to obtain the following application.

\begin{theorem}\label{prop:mult:operateur}
    Let $T$ be an operator on a Banach space $X$. Let $d\geq 1$ and $\Lambda\subset(0,+\infty)^d$ be a self-similar fractal with Hausdorff dimension $s>0$ satisfying the OSC. Then, the (continuous) family $(T_{\lambda})_{\lambda\in\Lambda}:=(\lambda(1) T\times \ldots\times \lambda(d) T)_{\lambda\in\Lambda}$ on $X^d$ does not admit any common upper frequently hypercyclic vectors.   
\end{theorem}
\begin{proof}
     We consider the sup norm on $X^d$. Let $0<a<b$ be such that $\Lambda\subset [a,b]^d$. We fix any $v\in S_X^d$, that is, $v=(v_1,\dots,v_d)$ with $\|v_i\|=1$ for $i=1,\dots, d$. Given $\lambda,\mu\in\Lambda$, let $1\leq j\leq d$ such that $\Vert \lambda-\mu\Vert_\infty=\vert \lambda(j)-\mu(j)\vert$.
Without loss of generality, we may assume that $\lambda(j)\leq \mu(j)$. Let $u\in X$ and $n\in N_{\lambda}(u,V)$ for $V=B_{\infty}(v,1/2)$. We have
\begin{align*}\|T_\lambda^n u-T_\mu^n u\|&\geq \|(\lambda(j)T)^n u_j-(\mu(j)T)^n u_j\|=\Big\vert\Big(\frac{\mu(j)}{\lambda(j)}\Big)^n-1\Big\vert \Vert(\lambda(j) T)^{n} u_j\Vert.
\end{align*}
Notice that
\[\Big|\frac{\mu(j)^n}{\lambda(j)^n}-1\Big| = \Big(\frac{\mu(j)}{\lambda(j)}\Big)^n-1 = \int_{1}^{\mu(j)/\lambda(j)} nx^{n-1}dx \geq n\Big(\frac{\mu(j)}{\lambda(j)}-1\Big)\geq \frac{n}{b}\vert \mu(j)-\lambda(j)\vert.\]
Since $\|v_j\|=1$ and $T_\lambda^nu\in V$, we get $\|(\lambda(j)T)^n u_j -v_j\|<1/2$, that is, $\|(\lambda(j)T)^n u_j\|>1/2$. Thus,
\[\|T_\lambda^n u-T_\mu^n u\|\geq \frac{n}{b}\vert\mu(j)-\lambda(j)\vert\Vert(\lambda(j) T)^{n} u_j\Vert\geq \frac{n}{2b}|\mu(j)-\lambda(j)|= \frac{n}{2b}\|\lambda-\mu\|_\infty.\]
From Proposition~\ref{prop:connect:new2} we satisfy condition \eqref{modif_teo:non-ex:eq} of Theorem~\ref{modif_teo:non-ex} with the non-decreasing function $\psi: \R_+\to\R_+$ given by $\psi(t)=t^r$, where $r=s$ in the case $s<1$ and $r=1/2$ in the case $s\geq 1$. Now, to be able to apply Theorem~\ref{modif_teo:non-ex}, we need to check that $\psi\in\Psi^*$. This is the case indeed, since conditions \ref{modif_teo:non-ex:ii} and \ref{modif_teo:non-ex:iii} are immediately satisfied, condition \ref{modif_teo:non-ex:v} is satisfied for any $t_0$ if we take $\eta<\frac{1}{r}-1$ and condition \ref{modif_teo:non-ex:best} is satisfied for every $\gamma\geq2$. This being granted, since from \cite[Theorem 4.14]{Mattila} we know that $0<\mathcal{H}^s(\Lambda)<+\infty$, it suffices to apply Theorem~\ref{modif_teo:non-ex} to infer the non-existence of common upper frequently hypercyclic vectors for the family $(T_{\lambda})_{\lambda\in\Lambda}$.
\end{proof}

The following consequence of Theorem~\ref{prop:mult:operateur} yields a solution to Question~\ref{Q3bis}.

\begin{proposition}\label{example:Cantor} 
Let $\mathcal{C}_{[a,b]}$ be the classical Cantor ternary set on the interval $[a,b]$ with $0<a<b<+\infty$. Then a family of multiples of a single operator $T$ indexed by $\mathcal{C}_{[a,b]}$ has no common upper frequently hypercyclic vector.
\end{proposition}

\begin{proof} 
The parameter set is well-known to be a self-similar fractal with Hausdorff dimension $s=\frac{\log(2)}{\log(3)}$ satisfying the OSC and with $\mathcal{H}^{s}(\Lambda)=(b-a)^s$. Thus, Theorem~\ref{prop:mult:operateur} ensures that a family of multiples of a single operator $T$ indexed by $\mathcal{C}_{[a,b]}$ has no common upper frequently hypercyclic vector.  
\end{proof}

We now focus our attention on the particular case of weighted backward shifts. Let $I\subset \R$ and recall that, for a family of weights $\big(w(x))_{x\in I}$ and all $n\in\N$, we define $f_n:I\to \R$ by $f_n(x)=\sum_{i=1}^n\log\big(|w_i(x)|\big)$. Let $T_\lambda=B_{w(\lambda(1))}\times \dots \times B_{w(\lambda(d))}$ for each $\lambda=(\lambda(1),\dots,\lambda(d))\in I^d$. In this setting, Lemma~\ref{prop:connect:newnew} allows us to obtain the following proposition in the same vein as Proposition \ref{prop:connect:new2}.

\begin{proposition}\label{prop:cas:shift}
    Given $d\geq 1$, let $\Lambda\subset (0,+\infty)^d$ be a self-similar fractal with Hausdorff dimension $s>0$ and satisfying the OSC. Consider a family $(B_{w(\lambda(1))}\times\cdots\times B_{w(\lambda(d))})_{\lambda\in\Lambda}\subset\L(X^d)$ of products of weighted shifts acting on $X=c_0$ or $X=\ell^p, 1\leq p<+\infty$ such that $\lambda\mapsto (B_{w(\lambda(1))}\times\cdots\times B_{w(\lambda(d))})(u)$ is continuous for all $u\in X^d$ and suppose that $\Lambda\subset I^d\subset (0,+\infty)^d$, where $I$ is an interval, and that there is a non-decreasing function $F:\N\to(0,+\infty)$ such that $\lim_{n\to\infty}F(n)=+\infty$ and for all $n\gg1$ and $x,y\in I$,
    \begin{equation}\label{eq:F:lips:inv:shift}
        F(n)|x - y|\leq |f_n(x)-f_n(y)|.
    \end{equation}
    Then there exist $c\in (0,+\infty)$ and $V\subset X$ open and non-empty such that 
    \begin{equation}
    \label{prop:key:eq:0:shift}
        \forall \, u\in X, \, \forall \, n\gg 1, \, \mathcal{H}^s(\{\lambda\in \Lambda : n\in N_\lambda(u,V)\})\leq \frac{c}{F(n)^s}.
    \end{equation}
\end{proposition}
\begin{proof}
  It suffices to apply Lemma~\ref{prop:connect:newnew} for the family $(T_\lambda)_{\lambda\in\Lambda}$ where $T_\lambda = B_{w(\lambda(1))}\times \cdots \times B_{w(\lambda(d))}$ for all $\lambda\in \Lambda$. 
  
  Let us choose, $v=e_0\times\ldots\times e_0$, $\delta=\frac{1}{2}$ and let us define $V=B(v,\delta)$ and $c_{u,v}=\frac{1}{2}$.
  Let $u\in X^d\setminus\{0\}$ and $n\geq 1$ be such that $n\in N_{\lambda}(u,V)$, if it exists, then 
  \begin{equation}\label{eq:corol:cas:shift}
      \Vert (B_{w(\lambda(1))}^n\times\cdots\times B_{w(\lambda(d))}^n)(u)-v\Vert = \|T_\lambda ^n u- v\| <\delta.
  \end{equation}

Let $i\in\{1,\dots, d\}$ be arbitrarily fixed. Projecting \eqref{eq:corol:cas:shift} first on the $i$-th coordinate and then on the term in $0$, we get $\big\vert\big(\prod_{k=1}^{n}w_k(\lambda(i))\big)u_n(i)-1\big\vert<\frac{1}{2}$, so $\frac{1}{2}\leq \big\vert\prod_{k=1}^{n}w_k(\lambda(i))\big\vert \vert u_n(i)\vert\leq\frac{3}{2}$. We will apply Lemma~\ref{prop:connect:newnew}. So let $\mu\in\Lambda$ with $n\in N_{\mu}(u,V)$. The same argument as above ensures $\frac{1}{2}\leq \big|\prod_{k=1}^{n}w_k(\mu(i))\big| |u_n(i)|\vert\leq\frac{3}{2}$. Thus,
\begin{align*}
F(n)|\mu(i)-\lambda(i)|&\leq |f_n(\mu(i))-f_n(\lambda(i))| \\
    &= \bigg\vert \log\bigg(\prod_{k=1}^{n}\vert w_k(\mu(i))\vert\bigg)-\log\bigg(\prod_{k=1}^{n}\vert w_k(\lambda(i))\vert\bigg)\bigg\vert \\
    &=\bigg\vert \log\bigg(\bigg\vert\prod_{k=1}^{n}w_k(\mu(i))\bigg\vert \vert u_n(i)\vert\bigg)-\log\bigg(\bigg\vert\prod_{k=1}^{n}w_k(\lambda(i))\bigg\vert \vert u_n(i)\vert\bigg)\bigg\vert \\
    &\leq \log(3).
    \end{align*}
This being valid for any coordinate $i$, we have 
    \[F(n)\Vert \lambda-\mu\Vert_\infty\leq \log(3).\]
From this we deduce that, if $F(n)\Vert \lambda-\mu\Vert_\infty> \log(3)$, then $n\notin N_{\mu}(u,V)$ and we can apply Lemma~\ref{prop:connect:newnew} with $c_{u,V}=\log(3)+1$. This finishes the proof.
\end{proof}

When looking for examples, it can be useful to have 
some tools to check that a family of weighted shifts satisfies the continuity assumptions needed to apply our results. We will say that a family of weights $(w(\lambda))_{\lambda\in\Lambda}\subset \ell^\infty$ is \emph{continuous} if the map $\lambda\mapsto w_k(\lambda)$ is continuous for every $k\in\N$. Moreover, this family of weights will be called \emph{uniformly bounded} if there exists $C>0$ such that for every $k\in\N$, $\sup_{\lambda\in\Lambda}\vert w_{k}(\lambda) \vert\leq C$.

\begin{proposition}\label{Prop:shift:continuous}
    Let $\Lambda\in\R$ be a compact subset and $X=c_0$ or $\ell^p$ with $1\leq p<+\infty$. Let also $(w(\lambda))_{\lambda\in\Lambda}$ be a continuous and uniformly bounded family of weights. Then the map $\lambda\mapsto B_{w(\lambda)}u$ is continuous for every $u\in X$.
\end{proposition}

\begin{proof}
We will prove this for $X=\ell^p$, the proof for $c_0$ being essentially the same. Let $\varepsilon>0$ and $u\in \ell^p$. Since $u\in\ell^p$, we can find $N\geq 1$ such that $\sum_{n\geq N}\vert u_n\vert^p<\frac{\varepsilon}{4C}$. By continuity of $\lambda\mapsto w_n(\lambda)$ for $1\leq n\leq N-1$, there exists $\delta>0$ such that, if $\vert\lambda-\mu\vert<\delta$, then $\sum_{n=1}^{N-1}\vert w_{n}(\lambda)-w_n(\mu)\vert\vert u_n\vert^p\leq \frac{\varepsilon}{2}$.
Hence, if $\vert\lambda-\mu\vert<\delta$, 
\begin{align*}
    \Vert B_{w(\lambda)}u-B_{w(\mu)}u\Vert_{p}^{p}&=\sum_{n\geq 1}\vert w_{n}(\lambda)-w_n(\mu)\vert\vert u_n\vert^p\\
    &\leq \sum_{n=1}^{N-1}\vert w_{n}(\lambda)-w_n(\mu)\vert\vert u_n\vert^p+\sum_{n\geq N}\vert w_{n}(\lambda)-w_n(\mu)\vert\vert u_n\vert^p\\
    &\leq \frac{\varepsilon}{2}+2C\sum_{n\geq N}\vert u_n\vert^p\\
    &\leq \varepsilon.
\end{align*}
\end{proof}

As we mentioned in the introduction, in the one-dimensional case $(T_\lambda)_{\lambda\in\Lambda}$ has a common upper frequently hypercyclic vector as soon as $F$ grows as $\log$ or slower. In the next example we verify that this condition is optimal in the sense that, if $F$ grows any faster than $\log$,  
then $(T_\lambda)_{\lambda\in\Lambda}$ has no common upper frequently hypercyclic vector when $\Lambda$ is an interval, which answers Question~\ref{Q3}.

\begin{example} \label{example:prod:comp:log:ufhc}
Take $\Lambda=[a,b]^s$ and $s\in\N$. Suppose that $(w(x))_{x\in [a,b]}$ is a continuous and uniformly bounded family of weights satisfying \eqref{eq:F:lips:inv:shift} for $F(t)=\big(\prod_{j=1}^{N-1}\log^{(j)}(t)\big)^{\frac{1}{s}}\big(\log^{(N)}(t)\big)^\beta$, for some $N\geq1$, $\beta>\frac{1}{s}$ and $t_0>\!\!>1$ and let $\psi:(t_0,+\infty)\to[1,+\infty)$ given by $\psi(t)=F(t)^s$. Then $(B_{w(\lambda(1))}\times\cdots\times B_{w(\lambda(d))})_{\lambda\in\Lambda}$ has no common upper frequently hypercyclic vector. Indeed, first notice that Proposition \ref{Prop:shift:continuous} guarantees that the induced backward shifts are all bounded. We aim to apply Theorem \ref{modif_teo:non-ex}, so let us first verify that $\psi\in\Psi^*$. As we noticed in Remark \ref{remark:t0:big}, there is no problem if $\psi$ is only defined on $(t_0,+\infty)$. Conditions \ref{modif_teo:non-ex:ii}, \ref{modif_teo:non-ex:iii} and \ref{modif_teo:non-ex:best} are easily checked. By some straightforward computations, we obtain
\begin{align*}
    \frac{\psi(t)}{\psi'(t)}&=\frac{\big(\prod_{j=1}^{N-1}\log^{(j)}(t)\big)\big(\log^{(N)}(t)\big)^{\beta s}}{\frac{1}{t}\Big(\beta s\big(\log^{(N)}(t)\big)^{\beta s-1}+\sum_{k=1}^{N-1}\big(\prod_{j=k+1}^{N-1}\log^{(j)}(t)\big)\big(\log^{(N)}(t)\big)^{\beta s}\Big)}\\
    &\sim t\log(t)\frac{\big(\prod_{j=2}^{N-1}\log^{(j)}(t)\big)\big(\log^{(N)}(t)\big)^{\beta s}}{\big(\prod_{j=2}^{N-1}\log^{(j)}(t)\big)\big(\log^{(N)}(t)\big)^{\beta s}}\\
    &\sim t\log(t)  ,      
\end{align*}
    thus \ref{modif_teo:non-ex:v} is satisfied, that is, $\psi\in\Psi^*$. The conclusion then follows from combining Proposition \ref{prop:cas:shift} with Theorem \ref{modif_teo:non-ex}.
\end{example}

Remark also that when the parameter set $\Lambda$ is $d$-dimensional with $d>1$, the previous example yields an upper bound for the optimal growth of $F$ ensuring the existence of common upper frequently hypercyclic vectors. Indeed, our previous example demonstrates that as soon as $F$ grows essentially faster than $\log^{\frac{1}{s}}$, such vectors cannot exist, what gives a partial answer to Question~\ref{Q3ter}.

\begin{example}
One can explicitly exhibit an admissible family of weights satisfying Example \ref{example:prod:comp:log:ufhc} by simply defining
    \[w_1(x)\cdots w_n(x)=\exp(x F(n)), \quad \forall \, x \in [a,b], \forall \, n\in\N,\]
    where $F=\psi^{1/s}$ with $\psi\in\Psi^*$ is the desired function. Of course we only need that the products $w_1(x)\ldots w_n(x)$ ``behave'' like $\exp(xF(n))$ asymptotically on $n$ for our argument to work. Let us prove that such a family of weighted shifts is uniformly bounded. To do so, it suffices to show that the sequence of differences $(F(n)-F(n-1))_n$ is bounded. Clearly, since $F$ is $C^1$, for all $n\in\mathbb{N}$ sufficiently large, there exists $c_n\in (n-1,n)$ such that $F(n)-F(n-1)=F'(c_n)$ by the Mean Value Theorem. According to condition \ref{modif_teo:non-ex:v} of Notation \ref{notation:F} and Remark \ref{Rem:FInfSeries}, which gives some constant $C>0$, we may write 
    \[ 
    0\leq F'(c_n) = \frac{1}{s}\psi^{1/s}(c_n)\frac{\psi'(c_n)}{\psi(c_n)} \leq \frac{1}{s}\psi^{1/s}(c_n)\frac{1}{(1+\eta)c_n}=\frac{1}{sC^{1/s}(1+\eta)}\frac{1}{c_n^{1-\frac{1}{s(1+\eta)}}}\xrightarrow{n\to+\infty}0,
    \] 
    for $s> \frac{1}{1+\eta}$. 
    This finishes the proof. 
\end{example}

\subsection{Frequent hypercyclicity}

In this section, we turn our attention to frequent hypercyclicity. By tweaking the techniques employed previously, we shall get analogous results on the non-existence of common frequently hypercyclic vectors.

\begin{theorem} \label{teo:non-exFHC}
Let $X$ be a separable $F$-space, $a<b$ be real numbers, $\Mes$ be a Borel measure  with $0<\Mes([a,b])<+\infty$ and consider a family $(T_\lambda)_{\lambda\in[a,b]}\subset \L(X)$ such that $\lambda\mapsto T_\lambda u$ is continuous for all $u\in X$ and $\psi\in\Psi$.
		Assume that, for all $u\in X$, there are a non-empty open set $V\subset X$ and $c>0$ such that
		\begin{equation}\label{teo:non-exFHC:eq}
		    \forall \, n\gg 1, \, \Mes(\{\lambda\in[a,b]:n\in N_{\lambda}(u,V)\})\leq \frac{c}{\psi(n)}.
		\end{equation}
        Then $\big(T_\lambda\big)_{\lambda\in [a,b]}\subset \L(X)$ has no common frequently hypercyclic vector.
        \end{theorem}
        \begin{proof}
        The proof follows steps very similar to that of Theorem \ref{modif_teo:non-ex}. We arbitrarily fix $u\in X$ and, from the hypothesis, we get $V\subset X$ open and non-empty and $c>0$ such that \eqref{teo:non-exFHC:eq} holds for all $n\gg1$. We choose a sequence $(k_l)_l\subset \N$ of positive integers that is fast growing enough so that $\sum_{n\geq1}\frac{1}{\sqrt{\psi(k_n)}}<+\infty$, $\sum_{n\geq 1}\frac{1}{k_n}<+\infty$ and $(\psi(k_l))_l \subset (1,+\infty)$ is increasing. This can easily be done recursively with the aid of condition \ref{modif_teo:non-ex:iii}. If we put $\alpha_l=\sqrt{\psi(k_l)}$ for all $l\in\N$, then condition \ref{modif_lemma:cond:alpha:k} of Lemma \ref{modif_lemma:bornes} is satisfied.
        Also, we can fix $N_0\geq t_0\in\N$ such that condition \ref{modif_lemma:cond:N0} of Lemma \ref{modif_lemma:bornes} is satisfied. Now we can apply Lemma \ref{modif_lemma:bornes} and find an increasing sequence $(N_l)_{l\geq 1}$ of positive integers as well as a decreasing sequence $(I_l)_{l\geq 1}$ of subsets of $[a,b]$ satisfying the assertions \ref{modif_lemma:cond:i}-\ref{modif_lemma:cond:iv} of Lemma \ref{modif_lemma:bornes}. By taking $I=\bigcap_{l\in\N}I_l$, we get $\Mes(I_l)\to \Mes(I)$ as $l\to+\infty$. Condition \ref{modif_lemma:cond:iv} of Lemma \ref{modif_lemma:bornes} guarantees that $I\neq \varnothing$. Given $\lambda_0\in I$, let us prove that $u\notin FHC(T_{\lambda_0})$. We shall check that 
		\begin{equation}\label{lemma:eq:lim}
			\frac{\#(N_{\lambda_0}(u,V)\cap[0,N_{l}])}{N_l+1}\xrightarrow{l\to+\infty}0.
		\end{equation}
        For all $l\geq 1$, since $\lambda_0\in I_l$, condition \ref{modif_lemma:cond:i} of Lemma \ref{modif_lemma:bornes} gives
        \[\#(N_{\lambda_0}(u,V)\cap[N_{l-1}+1,N_l])<\alpha_lk_l \implies \#(N_{\lambda_0}(u,V)\cap[N_{0}+1,N_{l}])<\sum_{i=1}^{l}\alpha_ik_i.\]
		Since $(\alpha_l)_l$ is increasing, for any $l\geq1$ we have
		\begin{align*}
		  \sum_{i=1}^l\alpha_ik_i 
			\leq 
            \alpha_l\sum_{i=1}^lk_i
            = \sqrt{\psi(k_l)}\sum_{i=1}^lk_i.
		\end{align*}
		Thus, using conditions \ref{modif_lemma:cond:ii} and \ref{modif_lemma:cond:iv} from Lemma \ref{modif_lemma:bornes}, just like in the proof of Theorem \ref{modif_teo:non-ex}, we get
		\begin{align*}
			\frac{N_l}{\psi(N_l)}
            \geq \frac{\eta}{c(1+\eta)}\sum_{i=1}^l k_i\Mes(I_{i-1}) 
            \geq \frac{\eta C_\alpha\Mes([a,b])}{2c(1+\eta)} \sum_{i=1}^lk_i.
		\end{align*}
		We can now use that $\psi$ is non-decreasing and \ref{modif_teo:non-ex:ii} of Notation \ref{notation:F} to obtain
		\begin{align*}
			N_l&\geq K\psi(N_l)\sum_{i=1}^lk_i\geq K\psi\bigg(K\psi(N_l)\sum_{i=1}^lk_i\bigg)\sum_{i=1}^lk_i\geq K\psi\bigg(K\sum_{i=1}^lk_i\bigg)\sum_{i=1}^lk_i\gtrsim \psi(k_l)\sum_{i=1}^lk_i,
		\end{align*}
        with $K=\frac{\eta C_\alpha\Mes([a,b])}{2c(1+\eta)}$. Consequently,
		\begin{align*}
			\frac{\#(N_{\lambda_0(u,V)}\cap[0,N_{l}])}{N_l+1}
			\leq \frac{N_0+1}{N_l+1}+\frac{1}{N_l+1}\sum_{i=1}^{l}\alpha_ik_i 
			\lesssim \frac{N_0+1}{N_l+1} 
            + \frac{\sqrt{\psi(k_l)}\sum_{i=1}^lk_i}{\psi(k_l)\sum_{i=1}^lk_i}
            \xrightarrow{l\to+\infty}0,
		\end{align*}
        thus proving \eqref{lemma:eq:lim}. This straightforwardly implies that $\underline{d}(N_{\lambda_0}(u,V))=0$ as we wanted to prove.
	\end{proof}
    
    In the following example, we determine the ``optimal'' growth rate of $F$ in the common frequent hypercyclicity case.
    
    \begin{example}
        Let us consider a continuous and uniformly bounded family of weights $(w(\lambda))_{\lambda\in [a,b]}$ satisfying \eqref{eq:F:lips:inv:shift} for the function $F:(t_0,+\infty)\to\R_+$ given by $F(t)=\big(\log^{(N)}(t)\big)^\beta$ with $\beta>0$, $N\geq 1$ and $t_0>\!\!>1$. First notice that $\psi=F\in \Psi$. Indeed, conditions \ref{modif_teo:non-ex:ii} and \ref{modif_teo:non-ex:iii} are clearly satisfied and one may check condition \ref{modif_teo:non-ex:v} as in Example~ \ref{example:prod:comp:log:ufhc}. Therefore, $\psi\in\Psi$. Then, to apply Theorem~\ref{teo:non-exFHC}, condition \eqref{teo:non-exFHC:eq} with $\Mes$ being the Lebesgue measure follows from Proposition~\ref{prop:cas:shift}. 
        We thus conclude that the induced family $(B_{w(\lambda)})_{\lambda\in[a,b]}$ has no common frequently hypercyclic vector.
\end{example}

Notice that the function $F$ in the previous example is ``barely'' unbounded: it goes to $+\infty$ as slow as one wants. Still, we have no frequently hypercyclic vector. As we shall see in the next section, as soon as $F$ is effectively bounded, then we will be able to obtain common frequently hypercyclic vectors for $(B_{w(\lambda)})_{\lambda\in[a,b]}$ under very natural assumptions. In this sense, $F$ being bounded is the optimal condition for the common frequent hypercyclicity when the parameter set is an interval.

Theorem~\ref{teo:non-exFHC} allows not only to recover a result of Bayart and Grivaux \cite{Baygrifrequentlyhcop,Bayhigh} on the non-existence of common frequently hypercyclic vectors for families of multiples, but also to derive a lot of other families that do not admit such common vectors. Indeed, in all the following examples, Remark~\ref{remarque:d=1} proves that \eqref{eq:F:lips:inv} is satisfied and we can apply Theorem~\ref{teo:non-exFHC} to conclude that these families do not admit any common frequently hypercyclic vector.

\begin{example}
  Let $T$ be an operator on a Fréchet space $X$, $\Lambda:=[a,b]\subset \R_{+}$ with $a<b$ and $T_\lambda=\lambda T$ for every $\lambda\in[a,b]$. Then, for every $\lambda<\mu\in\Lambda$, every $k\in\N$ and every $n\in N_{\lambda}(u,B(e_0,\frac{1}{2}))$, using the Bernoulli's inequality $(1+x)^n\geq 1+nx$ for all $x>-1$ and all $n\in\mathbb{N}$,
        \begin{align*}
            \Vert T_{\lambda}^{n}u-T_{\mu}^{n}u\Vert&= \Big\vert 1-\Big(\frac{\mu}{\lambda}\Big)^n\Big\vert \Vert T_{\lambda}^{n}u\Vert
            = \Big\vert \Big(1+\frac{\mu-\lambda}{\lambda}\Big)^n-1\Big\vert \Vert T_{\lambda}^{n}u\Vert
            \geq n \frac{\vert\mu-\lambda\vert}{\lambda}\Vert T_{\lambda}^{n}u\Vert
            \geq n \frac{\vert\mu-\lambda\vert}{2b}.
        \end{align*}
        \end{example}
        
        \begin{example}
        Consider $X$ a Fréchet sequence space. 
        \begin{enumerate}
        \item Let $\Lambda:=[a,b]\subset \R_{+}$ with $a<b$. For every $\lambda\in\R_+$, we define $w(\lambda)=(w_{n}(\lambda))_{n\in\N}:=(1+\frac{\lambda}{n})_{n\in\N}$. These operators are known to admit common frequently hypercyclic vectors when $\Lambda$ is any countable relatively compact subset in $(\frac{1}{p},+\infty)$ (see \cite[Example 2.20]{CEMM}). We prove that this is not the case anymore when $\Lambda$ is not countable.
        For every $\lambda<\mu\in\Lambda$ and every $k\in\N$, remark that $w_k(\mu)\geq w_k(\lambda)(1+\frac{\mu-\lambda}{n+b})$.
        Hence, for every $n\in N_{\lambda}(u,B(e_0,\frac{1}{2}))$, using the classical inequality $\prod_{k=1}^n(1+x_k)\geq 1+\sum_{k=1}^nx_k$ for all $x_k\geq 0$ and $n\in\mathbb{N}$, we get
        \begin{align*}
            \Vert B_{w(\lambda)}^{n}u-B_{w(\mu)}^{n}u\Vert
            &\geq \vert (B_{w(\lambda)}^{n}u)(0)-(B_{w(\mu)}^{n}u)(0)\vert\\
            &\geq  \bigg\vert\prod_{k=1}^{n}w_k(\lambda)-\prod_{k=1}^{n}w_k(\mu)\bigg\vert \vert u_n\vert\\
            &\geq \bigg\vert \prod_{k=1}^{n}\frac{w_k(\mu)}{w_k(\lambda)}-1\bigg\vert \bigg\vert\prod_{k=1}^{n}w_k(\lambda)\bigg\vert \vert u_n\vert\\
            &\geq \bigg\vert \prod_{k=1}^{n}\Big(1+\frac{\mu-\lambda}{k+b}\Big)-1\bigg\vert \bigg\vert\prod_{k=1}^{n}w_k(\lambda)\bigg\vert \vert u_n\vert\\
            &\geq \bigg\vert \bigg(1+(\mu-\lambda)\sum_{k=1}^{n}\frac{1}{k+b}\bigg)-1\bigg\vert \bigg\vert\prod_{k=1}^{n}w_k(\lambda)\bigg\vert \vert u_n\vert\\
            &\geq \vert \mu-\lambda\vert (\log(n+1+b)-\log(b+1))\bigg\vert\prod_{k=1}^{n}w_k(\lambda)\bigg\vert \vert u_n\vert\\
            &\geq \vert\mu-\lambda\vert \log\Big(1+\frac{n}{b+1}\Big)\bigg\vert\prod_{k=1}^{n}w_k(\lambda)\bigg\vert \vert u_n\vert\\
            &\geq \vert\mu-\lambda\vert \frac{1}{2}\log\Big(1+\frac{n}{b+1}\Big).
        \end{align*}
        \item Let $(v_n)_{n\in\N}$ be a sequence of non-negative numbers
        and $\Lambda:=[a,b]\subset \R_{+}$ with $a<b$. For every $\lambda\in\R_+$, we define $w(\lambda)=(w_{n}(\lambda))_{n\in\N}:=(v_{n}+\lambda)_{n\in\N}$ and we assume that $T_\lambda=B_{w(\lambda)}$ is bounded.
        Then, for every $\lambda<\mu\in\Lambda$ and every $k\in\N$, remark that $w_k(\lambda)\leq \frac{\Vert v\Vert_{\infty}+\lambda}{\Vert v\Vert_{\infty}+\mu}w_k(\mu)$.
        Hence, for every $n\in N_{\lambda}(u,B(e_0,\frac{1}{2}))$, using Bernoulli's inequality once more, we get
        \begin{align*}
            \Vert B_{w(\lambda)}^{n}u-B_{w(\mu)}^{n}u\Vert
            &\geq \bigg\vert 1-\bigg(\frac{\Vert v\Vert_{\infty}+\mu}{\Vert v\Vert_{\infty}+\lambda}\bigg)^n\bigg\vert \bigg\vert\prod_{k=1}^nw_k(\lambda)\bigg\vert \vert u_n\vert\\
            &\geq \bigg\vert \bigg(1+\frac{\mu-\lambda}{\Vert v\Vert_{\infty}+\lambda}\bigg)^n-1\bigg\vert \bigg\vert\prod_{k=1}^nw_k(\lambda)\bigg\vert \vert u_n\vert\\
            & \geq  n\frac{\vert\mu-\lambda\vert}{\Vert v\Vert_{\infty}+\lambda} \bigg\vert\prod_{k=1}^nw_k(\lambda)\bigg\vert \vert u_n\vert\\
            &\geq n\vert\mu-\lambda\vert\frac{1}{2(\Vert v\Vert_{\infty}+b)}.
        \end{align*}
        
    \end{enumerate}
    \end{example}
\section{Common FHC for uncountable families}\label{sec:CommonFHC:positive}

The study of common frequent hypercyclicity has been developed very slowly and is still not well understood. Bayart and Grivaux \cite{Baygrifrequentlyhcop,Bayhigh} first proved that an uncountable family of real multiples of a single operator can never admit common frequently hypercyclic vectors. Ansari and Le\'on-Müller's results on hypercyclicity have been generalized to the frequently hypercyclic setting showing that for every bounded operator $T$ on a complex Banach space $X$, every  $n\in\N$ and any $\lambda\in\mathbb{T}$, $FHC(T^n)=FHC(\lambda T)=FHC(T)$. These are extremely strong common frequent hypercyclicity results. Of course, this situation is a bit peculiar and we cannot expect to get such strong results with general families. However, Grivaux, Matheron and Menet \cite{GriMathMenorthogonality} gave another example of a strong result in the particular framework of weighted shifts (see Proposition \ref{GMM}). Having such vectors for a family of operators is much more difficult than having common hypercyclic vectors and even common $\mathcal{U}$-frequently hypercyclic vectors. Indeed, it is well-known that both the set of hypercyclic vectors and the set of $\mathcal{U}$-frequently hypercyclic vectors are sets of Baire second category. On the contrary, it has been proven by Moothathu \cite{MoothathuTwoRemarks} and Bayart and Ruzsa \cite{BayartRuzsa} that the set of frequently hypercyclic vectors for an operator is always a set of Baire first category. Hence, the interesting families in the hypercyclic and $\mathcal{U}$-frequent hypercyclic cases are necessarily uncountable whereas the question seems to be non-trivial even for two operators in the frequently hypercyclic case.
This motivated Charpentier, Ernst, Mestiri and Mouze \cite{CEMM} to study in detail both the existence and the non-existence of common frequently hypercyclic vectors for countable families of operators. This last article provided results and examples that could not have been reached by the preceding approaches that were aiming at extremely strong results, namely the equality of the sets of frequently hypercyclic vectors. In particular, the authors gave necessary and also sufficient conditions on families of multiples of operators to admit common frequently hypercyclic vectors. This provides necessary and sufficient conditions on the multiples for some particular families of operators and it also permitted to exhibit two frequently hypercyclic multiples of a weighted shift that do not share any frequently hypercyclic vector. However, in their article, the authors 
did not discuss sufficient conditions for uncountable families. This is what we are aiming to do in the first part of the present section. 
In the second part of this section, we disprove some natural weakenings of \cite[Proposition 6.9]{GriMathMenorthogonality}, which ensures that two weighted shifts share the same set of frequently hypercyclic vectors.

\subsection{Uncountable families}

Recall that we have already noticed in the previous section that one has no common frequently hypercyclic vectors when $F(n)\to+\infty$, no matter how slow. We are, essentially, going to prove that if $F(n)$ is bounded, then common frequently hypercyclic vectors may exist.

To prove our main theorem, we will need a now classical combinatorial lemma. This version is a direct consequence of \cite[Lemma~2.2]{CEMM}. 

\begin{lemma}\label{lemme-refi}
	For every $K>1$ and every countable family $(N_p(i))_{p\in \N}$, $i\in \N$, of increasing sequences of positive integers, there exists a countable family $(E_p(i))_{p\in \N}$, $i\in \N$, of sequences of subsets of $\N$ with positive lower density, such that for every $(p,i),(q,j)\in \N^2$ and every $(n,m)\in E_p(i)\times E_q(j)$,
	\begin{enumerate}
		\item if $(p,i)\neq (q,j)$, then $E_p(i)\cap E_q(j)=\varnothing$, \label{Lemme-refi0}
		\item $\min(E_p(i))\geq N_p(i)$, \label{Lemme-refi1}
		\item if $n\neq m$, then $|n-m|\ge \max(N_p(i),N_q(j))$, \label{Lemme-refi2}
		\item if $(p,i)\neq (q,j)$ and $n>m$, then $n\ge Km$. \label{Lemme-refi3}
	\end{enumerate}
\end{lemma}

Let us now state the main result of this section. This theorem allows us to construct uncountable families admitting common frequently universal vectors.

\begin{theorem}\label{thmgeneral}
	Let $X,Y$ be $F$-spaces, with $Y$ separable, let  $\Lambda\subset\R^d$ bounded
	and let $\mathcal{T}_\lambda=(T_{\lambda,n})_{n}$, $\lambda\in\Lambda$, be sequences of continuous linear operators from $X$ to $Y$. We assume that there exist a dense subset $Y_0$ of $Y$, mappings $S_{\lambda,n}:Y_0\to X$, $\lambda\in\Lambda$, $n\in\N_0$, such that for every $u\in Y_0$,
	
	\begin{enumerate}
		\item \label{c1} the series $\sum_{n\geq 0}T_{\lambda,m}(S_{\mu,m+n}(u))$ converges unconditionally, uniformly for $m\in \N_0$ and $\lambda,\mu\in \Lambda$;
		\item \label{c2} the series $\sum_{0\leq n\leq m}T_{\lambda,m}(S_{\mu,m-n}(u))$ converges unconditionally, uniformly for $m\in \N_0$ and $\lambda,\mu\in \Lambda$;
		\item \label{c3} the series $\sum_{n\geq 0}S_{\lambda,n}(u)$ converges unconditionally, uniformly for $\lambda\in \Lambda$;
		\item \label{c4} for every $\varepsilon>0$, there exists $\delta>0$ and $M\in\N$ so that for every $n\geq M$, and every $\lambda,\mu\in\Lambda$,
		\[\| \lambda-\mu\|_\infty\leq \delta \implies \Vert T_{\lambda,n}(S_{\mu,n}(u))-u\Vert \leq \varepsilon.\]
	\end{enumerate}
	Then, the family $\{\mathcal{T}_{\lambda};\ \lambda\in \Lambda\}$ admits a common frequently universal vector.
\end{theorem}

\begin{proof} Within the proof, the notation $\Vert \cdot \Vert$ will be indifferently used to denote an $F$-norm defining the topologies of $X$ or $Y$. Since $Y$ is separable, we can assume that $Y_0=\{y_0,y_1,\ldots\}$. Let $(\varepsilon_p)_{p\in \N}$ be a decreasing sequence of positive real numbers such that $p^2\veps_p\to0$ as $p\to +\infty$. Thanks to assumption (\ref{c4}), there exist a sequence $(\delta_p)_{p\in\N}$ of positive numbers and sequences of integers $(M_p)_{p\in\N}$ and $(I_p)_{p\in\N}$ such that, for each $p\in\N$, there exists $(\mu_{i}^{(p)})_{i=1}^{I_p}\in\Lambda^{I_p}$ satisfying
	\[
	\Lambda\subseteq \bigcup_{i=1}^{I_p}\{v\in\mathbb{R}^d : \Vert v-\mu_i^{(p)}\Vert_{\infty}<\delta_p \}
	\] 
	and, for every $\mu\in\mu_{i}^{(p)}+[-\delta_p,\delta_p]^d$ and every $n\geq M_p$,
	\begin{equation}\label{Eqprox}
		\Vert T_{\mu,n}(S_{\mu_{i}^{(p)},n}(y_p))-y_p\Vert \leq \varepsilon_p.	
	\end{equation}	
	Moreover, the other assumptions of the theorem imply the existence of a sequence $(N_p(i))_{p\in \N,i\in\N}$ with $N_{p}(i)\geq M_p$ for each $p\in\N$ and $i\in\N$, increasing with respect to $p\in\N$ such that for every $p\in \N$, every $1\leq i\leq I_p$, every finite set $F\subset \{N_p(i), N_p(i) +1,\ldots\}$, every $m\in \N$, every $q\in\{0,\ldots,p\}$, every $\lambda, \mu\in \Lambda$,
	\begin{multicols}{2}
		\begin{enumerate}[(i)]
			\item \label{eq1} $\displaystyle{\big\Vert \sum_{n\in F}T_{\lambda,m}(S_{\mu,m+n}(y_q))\big\Vert < \frac{\varepsilon_p}{\max_{0\leq k\leq p}(I_k)}}$;
			\item \label{eq2} $\displaystyle{\big\Vert \sum_{\substack{n\in F\\n\leq m}}T_{\lambda,m}(S_{\mu,m-n}(y_q))\big\Vert < \frac{\varepsilon_p}{\max_{0\leq k\leq p}(I_k)}}$;
			\item \label{eq3} $\displaystyle{\big\Vert \sum_{n\in F}S_{\lambda,n}(y_q)\big\Vert < \frac{\varepsilon_p}{I_q}}$;
		\end{enumerate}
	\end{multicols}
	Let $(E_p(i))_{i,p\in \N}$ be a sequence of sets given by Lemma \ref{lemme-refi} applied to the sequence $(N_p(i))_{i,p\in\N}$. We define
	\[
	x=\sum_{p\in \N}\sum_{i=1}^{I_p}\sum_{n\in E_{p}(i)}S_{\mu_{i}^{(p)},n}(y_p).
	\]
	One easily checks that $x \in X$. Indeed, since for every $p,i\in\N$,  $\min(E_{p}(i))\geq N_p(i)$, the assertion \ref{eq3} gives
	\[
	\sum_{p\in \N}\sum_{i=1}^{I_p}\bigg\Vert \sum_{n\in E_{p}(i)}S_{\mu_{i}^{(p)},n}(y_p)\bigg\Vert\leq \sum_{p\in \N}\sum_{i=1}^{I_p}\frac{\varepsilon_p}{I_p}=\sum_{p\in \N}\varepsilon_p< +\infty.
	\]
	Note that $x$ is even unconditionally convergent. Our goal is now to prove that $x$ is a frequently hypercyclic vector for each sequence $(T_{\lambda,n})_{n \in\N}$, $\lambda\in \Lambda$. Let $(r_q)_{q\in\N}$ be a sequence of positive real numbers with $r_q \to 0$ as $q\to +\infty$, to be chosen later. 
	Let us then fix $\lambda\in\Lambda$, $q\in \N$. Then, the definition of the covering of $\Lambda$ ensures the existence of an integer $1\leq i_{\lambda}\leq I_q$ so that
	\[\Vert \lambda-\mu_{i_{\lambda}}^{(q)}\Vert_{\infty}\leq\delta_q.\]
	Since the sets $E_p(i)$, $i,p\in \N$, have positive lower density, it is sufficient to prove that
	\begin{equation}\label{eqbut}
		\Vert T_{\lambda,m}(x) - y_q\Vert <r_q \text{ for every }m\in E_q(i_\lambda).
	\end{equation}
	Let then fix $m\in E_q(i_{\lambda})$. Using that $E_p(i)\cap E_q(j) =\varnothing$ if $(i,p)\neq (j,q)$ and that $x$ is unconditionally convergent in $X$, we can decompose $T_{\lambda,m}(x)$ as 
	\[
	T_{\lambda,m}(x)=T_{\lambda,m}(S_{\mu_{i_{\lambda}}^{(q)},m}(y_q))+\overbrace{\sum_{i=1}^{I_q}\sum_{\substack{n\in E_q(i)\\n\neq m}}T_{\lambda,m}(S_{\mu_{i}^{(q)},n}(y_q))}^{A_m}+\overbrace{\sum_{\substack{p\in \N\\p\neq q}}\sum_{i=1}^{I_p}\sum_{\substack{n\in E_p(i)\\n\neq m}}T_{\lambda,m}(S_{\mu_{i}^{(p)},n}(y_p))}^{B_{m}}.
	\]
	First, since $m\geq N_q(i_{\lambda})$, the inequality \eqref{Eqprox} gives
	\begin{equation}\label{term1}\left\Vert T_{\lambda,m}(S_{\mu_{i_{\lambda}}^{(q)},m}(y_q))-y_q\right\Vert < \varepsilon_q.
	\end{equation}
	
	We first estimate $\Vert A_m \Vert$. We have
	\[
	\Vert A_m \Vert \leq \sum_{i=1}^{I_q}\left(\bigg\Vert\sum_{\substack{n\in E_q(i)\\n<m}} T_{\lambda,m}(S_{\mu_{i}^{(q)},m-(m-n)}(y_q))\bigg\Vert + \bigg\Vert\sum_{\substack{n\in E_q(i)\\n> m}} T_{\lambda,m}(S_{\mu_{i}^{(q)},m+(n-m)}(y_q))\bigg\Vert\right).
	\]
	Given that $|n-m|\geq \max (N_q(i),N_q(j))$ for any $n\in E_q(i)$, $n\neq m$, \ref{eq1} and \ref{eq2} yield
	\begin{equation}\label{term2}
		\Vert A_m \Vert < \sum_{i=1}^{I_q}2\frac{\varepsilon_q}{I_q}=2\varepsilon_q.
	\end{equation}
	
	We now turn to estimating $\Vert B_{m} \Vert$. Again, by unconditional convergence of the series, we have
	\[
	\Vert B_{m} \Vert \leq \overbrace{\sum_{\substack{p\in\N\\p\neq q}}\sum_{i=1}^{I_p}\bigg\Vert\sum_{\substack{n\in E_p(i)\\n> m}} T_{\lambda,m}(S_{\mu_{i}^{(p)},m+(n-m)}(y_p))\bigg\Vert}^{B_m^1} + \overbrace{\sum_{\substack{p\in\N\\p\neq q}}\sum_{i=1}^{I_p}\bigg\Vert\sum_{\substack{n\in E_p(i)\\n<m}} T_{\lambda,m}(S_{\mu_{i}^{(p)},m-(m-n)}(y_p))\bigg\Vert}^{B_m^2}.
	\]
	We deal first with $B_m^1$. We have
	\begin{align}\label{Bm10}
		B_m^1 & \leq \sum_{p>q}\sum_{i=1}^{I_p}\bigg\Vert\sum_{\substack{n\in E_p(i)\\n> m}} T_{\lambda,m}(S_{\mu_{i}^{(p)},m+(n-m)}(y_p))\bigg\Vert \nonumber + \sum _{p < q}\sum_{i=1}^{I_p}\bigg\Vert\sum_{\substack{n\in E_p(i)\\n> m}} T_{\lambda,m}(S_{\mu_{i}^{(p)},m+(n-m)}(y_p))\bigg\Vert.
	\end{align}
	For $n\in E_p(i)$ with $p\neq q$, we have $|n-m|\geq \max(N_p(i),N_q(j))$, hence \ref{eq1} ensures
	\[\sum_{p>q}\sum_{i=1}^{I_p}\bigg\Vert\sum_{\substack{n\in E_p(i)\\n> m}} T_{\lambda,m}(S_{\mu_{i}^{(p)},m+(n-m)}(y_p))\bigg\Vert\leq \sum_{p>q}\sum_{i=1}^{I_p}\frac{\varepsilon_p}{I_p}=\sum_{p>q}\varepsilon_p\]
	and
	\[\sum _{p < q}\sum_{i=1}^{I_p}\bigg\Vert\sum_{\substack{n\in E_p(i)\\n> m}} T_{\lambda,m}(S_{\mu_{i}^{(p)},m+(n-m)}(y_p))\bigg\Vert\leq \sum _{p < q}\sum_{i=1}^{I_p}\frac{\varepsilon_q}{I_p}=q\varepsilon_q.\]
	This gives us
	\[B_{m}^{1}\leq q\varepsilon_q+\sum_{p>q}\varepsilon_p.\]
	Using the same arguments and \ref{eq2} instead of \ref{eq1}, we get the same inequality for $B_m^2$.
	
	Altogether, we obtain
	\[\Vert T_{\lambda,m}(x)-y_q\Vert \leq \varepsilon_q + 2\varepsilon_q+2q\varepsilon_q+2\sum_{p>q}\varepsilon_p:=r_q,\]
	which tends to $0$ as $q\to +\infty$.
\end{proof}

\begin{remark}  For $d=1$, the Costakis-Sambarino condition \eqref{c4} in Theorem~\ref{thmgeneral} could have been stated like in the original criterion, as we have done in Theorem~\ref{CS:crit:F} with condition \ref{CS:F:2}. This version is technically stronger than the choice made here. However, the current statement holds in the $d$-dimensional setting for all $d\geq 1$.
\end{remark}

The conditions in the previous theorem are quite demanding at first sight, but it seems more or less inevitable for many reasons. As we already mentioned before, common frequent hypercyclicity is a very strong notion, impossible to reach for uncountably many multiples of a single operator as noticed by Bayart and Grivaux \cite{Baygrifrequentlyhcop} in 2004 and even for two different positive multiples of a well-chosen operator as noticed by Charpentier, Ernst, Mestiri and Mouze \cite{CEMM} in 2021. 
Here, our theoretical result aims at finding uncountable families of operators sharing a common frequently hypercyclic vector. The comparison with the previous studies suggests that the conditions to admit such a vector 
are expected to be quite strong. As mentioned in the previous sections, once we associate an increasing function $F:\N\to\R_+$ with the family of interest, its growth determines whether a positive result is possible or not. For common frequent hypercyclicity in particular, $F$ must be bounded, and this imposes the addition of the strong Costakis-Sambarino condition \eqref{c4} in our main result Theorem \ref{thmgeneral}. 

Actually, there are already some very specific results yielding the existence of common frequently hypercyclic vectors in a strong sense. 
We already mentioned some of them previously but we might also mention another one inducing the existence of uncountable families of operators from Grivaux, Matheron and Menet that is specific to weighted shifts.

\begin{proposition}[{\cite[Proposition 6.9]{GriMathMenorthogonality}}]\label{GMM}
	Let $w$ and $w'$ be two weight sequences such that $B_{w}$ and $B_w'$  are frequently hypercyclic on $\ell^p$. If $\frac{w_1\ldots w_n}{w_1'\ldots w_n'}$ has a non-zero limit as $n\to+\infty$, then $B_{w}$ and $B_{w'}$ share the same set of frequently hypercyclic vectors.
\end{proposition}

One may remark that the frequently hypercyclic version of Le\'on-Müller's theorem and the previous result are extremely specific. They apply to unimodular multiples of a single operator for the first one, and to weighted shifts satisfying a very strict condition for the second one. Indeed, this is due to the strength of their results. Their families do not simply have common frequently hypercyclic vectors, their sets of frequently hypercyclic vectors coincide. On the other hand, our result even allows to detect families having common frequently hypercyclic vectors and may be used for general uncountable families of operators.

We now state a simple corollary for weighted shifts, offering a direct way to obtain concrete applications of Theorem~\ref{thmgeneral}.

    \begin{corollary}\label{cor_encadrement}
		Let $\Lambda:=[l,L]\subset (0,+\infty)$.
		Assume that the family $(B_{w(\lambda)})_{\lambda\in\Lambda}$ satisfies condition (\ref{c4}) of Theorem~\ref{thmgeneral} and is such that there exist $a>1$ and $C>1$ so that, for every $\lambda\in\Lambda$, every $n\in\N$ and every $k\in\N_0$, \[\frac{1}{C} a^n\leq\prod_{i=k+1}^{k+n}\vert w_i(\lambda)\vert\leq C a^{n}.\]
		Then $\bigcap_{\lambda\in\Lambda}FHC(B_{w{(\lambda)}})\neq \varnothing$.
	\end{corollary}
	\begin{proof}
		Let us check the assumptions of Theorem \ref{thmgeneral}. We begin with
		\begin{align*}
			\sum_{n\geq 0}\Vert T_{\lambda,m}(S_{\mu,m+n}(u))\Vert&= \sum_{n\geq 0}\bigg\Vert \sum_{l\geq0}\frac{\prod_{k=n+l+1}^{m+n+l}w_{k}(\lambda)}{\prod_{k=l+1}^{m+n+l}w_{k}(\mu)}u_l e_{l+n}\bigg\Vert\leq\sum_{n\geq 0}\frac{Ca^m}{\frac{1}{C}a^{n+m}}\Vert u\Vert\leq\sum_{n\geq 0}C^2 a^{-n}\Vert u\Vert.
		\end{align*} 
		Thus, \eqref{c1} is satisfied. Now we have
		\begin{align*}
			\sum_{0\leq n\leq m}\Vert T_{\lambda,m}(S_{\mu,m-n}(u))\Vert&=\sum_{0\leq n\leq m}\bigg\Vert \sum_{l\geq0}\frac{\prod_{k=l+1}^{m+l}w_{k}(\lambda)}{\prod_{k=l+n+1}^{m+l}w_{k}(\mu)}u_{l+n} e_{l}\bigg\Vert\\
			&\leq\sum_{0\leq n\leq \min(m,d(u))}\frac{Ca^m}{\frac{1}{C}a^{m-n}} \Vert u\Vert\\
			&\leq\sum_{0\leq n\leq d(u)}C^2 a^n\Vert u\Vert. 
		\end{align*}
		Thus, \eqref{c2} is satisfied. Finally,
		\[\sum_{n\geq 0}\Vert S_{\lambda,n}(u)\Vert\leq \sum_{n\geq 0}Ca^{-n}\Vert u\Vert.\] 
		Thus, \eqref{c3} is satisfied. Hence, $\bigcap_{\lambda\in\Lambda}FHC(B_{w_{(\lambda)}})\neq\varnothing$ as we wanted.
	\end{proof}
		
	We illustrate the interest of Theorem~\ref{thmgeneral} by constructing the simplest example that comes to mind of a family of weighted shifts admitting common frequently hypercyclic vectors as an application of Corollary~\ref{cor_encadrement}.
	
	\begin{example}\label{exsimple}
		Let $\Lambda:=[l,L]\subset (0,+\infty)$ and $a>1$.
		Let us define $B_{w(\lambda)}$ so that $w_{2n}(\lambda)=\frac{a}{\lambda}$ and $w_{2n+1}(\lambda)=a\lambda$.
		Then $\bigcap_{\lambda\in\Lambda}FHC(B_{w(\lambda)})\neq \varnothing$. Indeed, it suffices to check that the conditions of the previous proposition are satisfied. 
		Remark first that one may express the product of the weights, for $i<j$, as 
		\[\prod_{k=i+1}^{j}w_k(\lambda)=\begin{cases}
			a^{j-i},&\text{ if }i\text{ and }j\text{ are odd},\\
			\frac{a^{j-i}}{\lambda},&\text{ if }i\text{ is odd and }j\text{ is even},\\
			a^{j-i},&\text{ if }i\text{ and }j\text{ are even},\\
			\lambda a^{j-i},&\text{ if }i\text{ is even and }j\text{ is odd}.
		\end{cases}\] This yields
		\[\min\Big(1,l,\frac{1}{L}\Big)a^n\leq \prod_{k=i+1}^{i+n}w_{k}(\lambda)\leq \max\Big(1,L,\frac{1}{l}\Big)a^n.\]
		Moreover, let $\varepsilon>0$ and choose $0<\delta<\frac{l\varepsilon}{\Vert u\Vert}$.
		Then \eqref{c4} of Theorem \ref{thmgeneral} is satisfied. Hence, we get $\bigcap_{\lambda\in\Lambda}FHC(B_{w(\lambda)})\neq\varnothing$. 
	\end{example}
	
	This simple example is not the best one to prove the usefulness of Theorem~\ref{thmgeneral}, as one may have remarked that it could have already been seen as a consequence of Proposition~\ref{GMM} coupled with Ansari's theorem for frequently hypercyclic operators. Indeed, by Ansari, for every $\lambda\in\Lambda$, $FHC(B_{w(\lambda)})=FHC(B_{w(\lambda)}^{2})$ and $B_{w(\lambda)}^{2}=a^2B$. Hence all these operators share the same set of frequently hypercyclic vectors since their squares are equal. 
	
	However, we give another example for which we cannot make use of previously known results. To give such an example, we need to avoid having the product of a fixed number of consecutive weights giving equal results, as it was the case in Example~\ref{exsimple}. At the same time, we want to stay ``close'' to some multiple of the unweighted backward shift. The idea is to split the weights in blocks of increasing length and to make the product of consecutive weights in these blocks equal to $\lambda$ to the power the length of the block.

	\begin{example} 
		To follow the idea previously mentioned, we construct the sequence of weights $w(\lambda)=(w_0(\lambda),\mathcal{B}_1,\mathcal{B}_2,\dots)$ with blocks $\mathcal{B}_n=[\ldots]$ of length $2n$, $n=1,2,\dots$, after the initial weight $w_0(\lambda)$ in the following spirit:
		\[
		w(\lambda)=\bigg(1,\Big[\frac{3}{\lambda},3\lambda\Big],\bigg[\frac{3}{\sqrt{\lambda}},3\lambda,\frac{3}{\lambda},3\sqrt{\lambda}\bigg],\bigg[\frac{3}{\sqrt[3]{\lambda}},3\sqrt{\lambda},\frac{3}{\lambda},3\lambda,\frac{3}{\sqrt{\lambda}},3\sqrt[3]{\lambda}\bigg],\dots  \bigg).
		\]
		This leads to the definition
		\[w_0(\lambda)=1, \quad \forall \, n\in\mathbb{N}, \, w_{n(n+1)+k}(\lambda)= \begin{cases}
			3\lambda^{\frac{(-1)^k}{n+2-k}},&\hbox{ for }k=1,2,\dots,n+1,\\
			3\lambda^{\frac{(-1)^k}{k-(n+1)}},&\hbox{ for }k=n+2,n+3,\dots,2(n+1).
		\end{cases}
		\] 
		By an easy calculation we compute
		\[\forall \, n\in\mathbb{N}, \, \prod_{j=1}^{n(n+1)+k}w_j(\lambda)=
		\begin{cases}
			3^{n(n+1)},&\hbox{ for }k=0,\\
			3^{n(n+1)+k}\lambda^{\sum_{l=1}^k\frac{(-1)^l}{n+2-l}},&\hbox{ for }k=1,2,\dots,n+1,\\
			\displaystyle\prod_{j=1}^{n(n+1)+2n+2-k}w_j(\lambda),&\hbox{ for }k=n+2,n+3,\dots,2n+1.
		\end{cases} \]
		Observe that, for all $n\geq 0$ and $k=1,\dots,n+1$, we have 
		\[
		\sum_{l=1}^k\frac{(-1)^l}{n+2-l}=(-1)^n\sum_{j=n+2-k}^{n+1}\frac{(-1)^j}{j}.
		\]
		Let $1\leq m\leq n$. Set $P_m(t)=\sum_{j=m}^{n+1}\frac{(-1)^j}{j}t^j$. We derive
		\[
		\begin{array}{rcl}\displaystyle\bigg\vert \sum_{j=m}^{n+1}\frac{(-1)^j}{j} \bigg\vert=\vert P_m(1)-P(0) \vert=
			\bigg\vert\sum_{j=m}^{n+1} \int_0^1(-1)^jt^{j-1}dt \bigg\vert&=&\displaystyle  \bigg\vert \int_0^1(-1)^mt^{m-1}\frac{1-(-t)^{n+2-m}}{1+t}dt \bigg\vert\\&\leq&\displaystyle \int_0^1\frac{t^{m-1}}{1+t}dt+\int_0^1\frac{t^{n+1}}{1+t}dt\\&\leq&\displaystyle 2\int_0^1\frac{1}{1+t}dt\\&\leq&2\log(2).
		\end{array}
		\]
		Let $\Lambda:=[a,b]\subset (1,+\infty)$. We deduce that for every $\lambda\in \Lambda$ and every $k,n\in\mathbb{N}$,
		\[
		b^{-2\log(2)} 3^n\leq\lambda^{-2\log(2)}3^n\leq\prod_{i=k+1}^{k+n}w_i(\lambda)\leq \lambda^{2\log(2)} 3^n\leq b^{2\log(2)} 3^n.
		\]
		Set $\eta=2\log(2)>1$. Let $\varepsilon>0$. Then, for $0<\delta<\frac{\varepsilon}{\eta\Vert u\Vert}a^{\eta}b^{1-\eta}$, property \eqref{c4} is satisfied. Hence we get $\bigcap_{\lambda\in\Lambda}FHC(B_{w(\lambda)})\neq\varnothing$.
	\end{example}

	\subsection{Close operators with distinct sets of frequently hypercyclic vectors}\label{subsec:bourdon-feldman}

	We now turn to a stronger form of frequent hypercyclicity.
	We could be tempted to modify slightly the result of Grivaux, Matheron and Menet stated in Proposition~\ref{GMM} to get a more general one. A first attempt could be to replace the unique limit by a finite number of limit points. We provide now a counter-example to such a statement.
	
	\begin{proposition}\label{Prop2ws}
		There exist two frequently hypercyclic weighted shifts $B_w$ and $B_{w'}$ on $\ell^p$ so that $\frac{\prod_{k=1}^{n}w_k}{\prod_{k=1}^{n}w'_{k}}$ has two different non-zero limit points as $n\to+\infty$ and $B_w$ and $B_{w'}$ do not share the same set of frequently hypercyclic vectors. 
	\end{proposition}
	
	\begin{proof}
		Let $w_k=3$ for every $k\in\N$ and $x\in FHC(B_w)$. Let us now define $w'$. We set 
		\[
		n_0=\min\{j\in\mathbb{N} : 3^jx_j\leq 1/2 \text{ or }3^jx_j\geq 3/2\}. 
		\]
		Clearly this positive integer exists since $x$ is a frequently hypercyclic vector for $B_w$. For every $n\geq n_0$, if $\frac{1}{2}< \big(\prod_{k=1}^{n+1}w_k \big)x_{n+1}<\frac{3}{2}$ but $\frac{1}{2}\geq \prod_{k=1}^{n}w_k x_{n}$ or $\big(\prod_{k=1}^{n}w_k\big) x_{n}\geq\frac{3}{2}$, then we define $w'_{n+1}=3w_{n+1}=9$. On the other hand, if $\frac{1}{2}< \big(\prod_{k=1}^{n}w_k\big) x_{n}\leq\frac{3}{2}$ but $\frac{1}{2}\geq \big(\prod_{k=1}^{n+1}w_k\big) x_{n+1}$ or $\big(\prod_{k=1}^{n+1}w_k\big) x_{n+1}\geq\frac{3}{2}$, then we define $w'_{n+1}=\frac{w_{n+1}}{3}=1$. Finally, in the other cases, we simply set $w'_{n+1}=w_{n+1}=3$. This way, the definition of $w'$ gives, for all $n\geq n_0$, 
		\[\frac{\prod_{k=1}^{n}w_k}{\prod_{k=1}^{n}w'_{k}}=\begin{cases}
			\frac{1}{3}& \text{ when }\frac{1}{2}\leq \big(\prod_{k=1}^{n}w_k\big) x_{n}\leq\frac{3}{2}\\
			1& \text{ when }\frac{1}{2}\geq \big(\prod_{k=1}^{n}w_k\big) x_{n}\text{ or }\big(\prod_{k=1}^{n}w_k\big) x_{n}\geq\frac{3}{2}
		\end{cases}
		\]
		Remark first that $\sum_{n\in\N}\frac{1}{\prod_{k=1}^{n}w'_k}\leq \sum_{n\in\N}\frac{1}{\prod_{k=1}^{n}w_k}<+\infty$, thus $B_{w'}$ is also frequently hypercyclic.
		Since $x\in FHC(B_w)$, both of the previous situation occur infinitely many times and then, $\frac{\prod_{k=1}^{n+1}w_k}{\prod_{k=1}^{n+1}w'_{k}}$ has two different non-zero limit points as $n\to+\infty$.
		Let us now prove that $x\notin HC(B_{w'})$. It suffices to prove that, for every $n\in\N$, $\| B_{w'}^{n}(x)-e_0\|\geq \frac{1}{2}$. Let $n\in\N$. If $n$ is so that $\frac{1}{2}\leq \big(\prod_{k=1}^{n}w_k \big)x_{n}\leq \frac{3}{2}$, then 
		\[\| B_{w'}^{n}(x) - e_0\| \geq \bigg\vert \bigg(\prod_{k=1}^{n}w'_k\bigg) x_{n}-1\bigg\vert=\bigg\vert 3\bigg(\prod_{k=1}^{n}w_k\bigg) x_{n}-1\bigg\vert \geq \frac{1}{2}.\] 
		If $n$ is so that either $\frac{1}{2}\geq\big(\prod_{k=1}^{n}w_k\big) x_{n}$ or $\big(\prod_{k=1}^{n}w_k \big) x_{n}\geq \frac{3}{2}$, then 
		\[\| B_{w'}^{n}(x) - e_0\|\geq \bigg\vert \bigg(\prod_{k=1}^{n}w'_k\bigg) x_{n}-1\bigg\vert = \bigg\vert \bigg(\prod_{k=1}^{n}w_k \bigg) x_{n}-1\bigg\vert \geq \frac{1}{2}.\] 
		This proves the proposition.
	\end{proof}

    It is worth noting that, although the sets of frequently hypercyclic vectors of these two operators are different, they nevertheless share some common frequently hypercyclic vectors. Indeed, one may easily check that these two operators satisfy Corollary~2.4 from \cite{CEMM}.
	
	One may also imagine that, if the sequence of quotients of products of the weights are converging on a big subset $A\subset \N$, then the frequently hypercyclic vectors could still remain the same. In fact, under some regularity assumptions on $A$, this is not the case. 

\begin{proposition}\label{Prop_Two_Shifts}
		 Let $w$ and $w'$ be positive weights such that $B_w$ and $B_{w'}$ are two bounded frequently hypercyclic weighted shifts satisfying that there is $A\subset \N$ infinite with $A^c$ infinite such that $\frac{\prod_{k=1}^{n}w_k}{\prod_{k=1}^{n}w'_{k}}$ tends to $a\ne 0$ as $n$ tends to infinity in $A$ and that there exists $\eta>0$ with $\vert a\vert-\eta>0$ such that $\left\vert\frac{\prod_{k=1}^{n}w_k}{\prod_{k=1}^{n}w'_{k}}-a\right\vert>\eta$ for $n\in A^c$ big enough. Assume that for every finite sequence $(b_i)_{i=0}^{k}\subset\{0,1\}$ containing at least one zero, 
		\begin{equation}\label{H1}
			\underline{d}(\{n\in\N: \forall \, 0\leq i\leq k, n+i\in A\text{ if, and only if, }b_i=0\})>0.
		\end{equation}
		Then, there exists $x\in FHC(B_w)$ such that $x\notin HC(B_{w'})$. 
	\end{proposition}
	
	\begin{proof}	
			We are going to construct $x\in FHC(B_w)$ with $x\notin HC(B_{w'})$. Let us consider a dense set of sequences $(y_p)_{p\in\N}\subset c_{00}$ and a sequence $(\varepsilon_{p})_{p\in\N}$ such that $\sum_{p\in\N}\varepsilon_{p}<+\infty$. Let $\gamma\leq \min(\frac{2}{2+3(\vert a\vert +\eta)/\eta},\frac{1}{2})$. 
			To each $y_p\in c_{00}$, we may associate a finite sequence $b_p$ of elements of $\{0,1\}$ as follows: for every $i=0,\ldots, d(y_p)+1$,   
			\[
			b_{p,i}=\begin{cases}
				0,&\text{ if }\ \vert y_{p,i}\prod_{k=1}^{i}w_k -1\vert >\gamma ,\\
				1,&\text{ if }\ \vert y_{p,i}\prod_{k=1}^{i}w_k -1\vert\leq\gamma.
			\end{cases}
			\]
			Observe that $b_{p,d(y_p)+1}=0$. Since $B_w$ is a frequently hypercyclic weighted shift, it satisfies in particular the Frequent Hypercyclicity Criterion, hence for every $p\in\N$, there exists $N_p\in\N$ so that, for every finite set $F\subset [N_p,+\infty)$ and every $1\leq i\leq p$, 
			\[
			\bigg\Vert\sum_{n\in F}B_{w}^{n}(y_i)\bigg\Vert+\bigg\Vert\sum_{n\in F}S_{w}^{n}(y_i)\bigg\Vert<\varepsilon_{p}.
			\]
			Let us also define, for each positive integers $p$, the set 
			\[B_p:=\{n\in\N: \forall \, 0\leq i\leq d(y_p)+1, n+i\in A\text{ if, and only if, }b_{p,i}=0\}.\]
			
			Let $(E_p)_{p\in\N}$ be the sequence of disjoint sets that we obtain when we apply \cite[Lemma 2.2]{MarMenPuigDFHCPshifts} with the sequence of sets $(B_p)_{p\in\N}$, the sequence of integers $(N_p)_{p\in\N}$ and $Q=1$. Then $E_p\subset B_p$, $\min(E_p)\geq N_p$ and, for every $n\in E_p$ and $m\in E_q$ with $n\neq m$, we have $\vert n-m\vert \geq N_p+N_q$. By a now classic construction, the vector $x$ defined as
			\[
			x=\sum_{p\in\N}\sum_{n\in E_p} S_w^n(y_p)
			\]
			is frequently hypercyclic for $B_w$.

			Let us now prove that $x\notin HC(B_{w'})$. It suffices to prove that, for every $n$ big enough, the following inequality holds: 
			\[\big\|B_{w'}^{n}(x) - {\textstyle\frac{1}{a}e_0}\big\|\geq \frac{\gamma}{2\vert a\vert}.\]
			Let $n\in\N$. First, if $n\notin \bigcup_p E_p+[0, d(y_p)]$, by definition of $x$ we observe that $x_n=0$, which yields 
			\[\big\|B_{w'}^{n}(x) - {\textstyle\frac{1}{a}e_0}\big\|\geq\Big\vert\frac{1}{a}\Big\vert\geq \frac{\gamma}{2\vert a\vert }\] 
			by the hypothesis on $\gamma$. Then, we have to deal with integers $n$ such that $n\in \bigcup_p E_p+[0,d(y_p)]$ i.e., such that there exists a positive integer $p$ such that $n=m+i$ with $m\in E_p$ and $ 0\leq i\leq d(y_p)$). 
			Two cases occur: either $n\notin A$ or $n\in A$.

            \
			
			\noindent \textbf{Case $n\notin A$.} By definition of $E_p\subset B_p$, we have $\vert \big(\prod_{k=1}^{i}w_k\big) y_{p,i}-1\vert\leq\gamma$. Let us recall that, by construction of $x$, $x_n\prod_{k=i+1}^{n}w_k=y_{p,i}$. Thus, we get
					
			\begin{align*}
				\big\|B_{w'}^{n}(x) - {\textstyle\frac{1}{a}e_0}\big\|
                &\geq \bigg\vert\bigg(\prod_{k=1}^{n}w'_k\bigg) x_n-\frac{1}{a}\bigg\vert\\
				&\geq \frac{1}{\vert a\vert}\bigg\vert a\frac{\prod_{k=1}^{n}w'_k}{\prod_{k=1}^{n}w_k}\bigg(\prod_{k=1}^{n}w_k\bigg)x_n-1\bigg\vert\\
				&\geq\frac{1}{\vert a\vert}\bigg\vert  a\frac{\prod_{k=1}^{n}w'_k}{\prod_{k=1}^{n}w_k}\bigg(\prod_{k=1}^{i}w_k\bigg)y_{p,i}-1\bigg\vert\\
				&\geq\frac{1}{\vert a\vert}\bigg\vert  a\frac{\prod_{k=1}^{n}w'_k}{\prod_{k=1}^{n}w_k}-1\bigg\vert\bigg\vert\bigg(\prod_{k=1}^{i}w_k\bigg)y_{p,i}\bigg\vert- \frac{1}{\vert a\vert}\bigg\vert\bigg(\prod_{k=1}^{i}w_k\bigg)y_{p,i}-1\bigg\vert\\&\geq\frac{1}{\vert a\vert}\bigg\vert  a\frac{\prod_{k=1}^{n}w'_k}{\prod_{k=1}^{n}w_k}-1\bigg\vert(1-\gamma)- \frac{\gamma}{\vert a\vert}.
			\end{align*}
			Moreover, since $n\notin A$, we have $\left\vert\frac{\prod_{k=1}^{n}w_k}{\prod_{k=1}^{n}w'_{k}}-a\right\vert>\eta$ if $n$ is large enough. 
                
            Therefore, we deduce 
			\[
			\left\vert a\frac{\prod_{k=1}^{n}w'_k}{\prod_{k=1}^{n}w_k}-1\right\vert>\eta\left\vert\frac{\prod_{k=1}^{n}w'_k}{\prod_{k=1}^{n}w_k}\right\vert\quad\hbox{ and }\quad
            \left\vert a\frac{\prod_{k=1}^{n}w'_k}{\prod_{k=1}^{n}w_k}-1\right\vert\geq 1-\vert a\vert\left\vert\frac{\prod_{k=1}^{n}w'_k}{\prod_{k=1}^{n}w_k}\right\vert,
			\]
            
           which implies 
			\[
            \vert a\vert\left\vert a\frac{\prod_{k=1}^{n}w'_k}{\prod_{k=1}^{n}w_k}-1\right\vert+\eta\left\vert a\frac{\prod_{k=1}^{n}w'_k}{\prod_{k=1}^{n}w_k}-1\right\vert\geq \vert a\vert\eta\left\vert\frac{\prod_{k=1}^{n}w'_k}{\prod_{k=1}^{n}w_k}\right\vert+\eta-\eta \vert a\vert\left\vert\frac{\prod_{k=1}^{n}w'_k}{\prod_{k=1}^{n}w_k}\right\vert=\eta.
            \]
            We get 
            \[ 
            \left\vert a\frac{\prod_{k=1}^{n}w'_k}{\prod_{k=1}^{n}w_k}-1\right\vert\geq\frac{\eta}{\vert a\vert+\eta}.
			\]

			Hence we get, thanks to the inequality $\gamma\leq \frac{2}{2+3(\vert a\vert+\eta)/\eta}$,
			\begin{equation*}
				\big\|B_{w'}^{n}(x) - {\textstyle\frac{1}{a}e_0}\big\|
				\geq\frac{1}{\vert a\vert} \frac{\eta}{\vert a\vert+\eta}(1-\gamma)- \frac{\gamma}{\vert a\vert}
				\geq\frac{\gamma}{2\vert a\vert}.
			\end{equation*}
			
			\noindent \textbf{Case $n\in A$.} Recall again that
			$\big(\prod_{k=i+1}^{n}w_k\big) x_n=y_{p,i}$ and $\vert\big(\prod_{k=1}^{i}w_k\big) y_{p,i}-1\vert>\gamma$. Since, for $n$ big enough, by hypothesis, $\left\vert a\frac{\prod_{k=1}^{n}w_k'}{\prod_{k=1}^{n}w_k}-1\right\vert\leq\varepsilon$ with $0<\varepsilon\leq \frac{\gamma}{2(1+\gamma)}<\frac{\gamma}{2(1-\gamma)}$, we deduce that
			\begin{align*}
				\big\|B_{w'}^{n}(x) - {\textstyle\frac{1}{a}e_0}\big\| 
                & \geq\bigg\vert \bigg(\prod_{k=1}^{n}w_k' \bigg) x_n -\frac{1}{a}\bigg\vert 
				\geq\bigg\vert \frac{\prod_{k=1}^{n}w_k'}{\prod_{k=i+1}^{n}w_k} y_{p,i} -\frac{1}{a}\bigg\vert
				\geq\bigg\vert \frac{\prod_{k=1}^{n}w_k'}{\prod_{k=1}^{n}w_k} \bigg(\prod_{k=1}^{i}w_k \bigg) y_{p,i} -\frac{1}{a}\bigg\vert.
			\end{align*}

\noindent We are going to prove that 
\[ 
\bigg\vert \frac{\prod_{k=1}^{n}w_k'}{\prod_{k=1}^{n}w_k} \bigg(\prod_{k=1}^{i}w_k \bigg) y_{p,i} -\frac{1}{a}\bigg\vert \geq \frac{\gamma}{2\vert a\vert}.
\]
To do this, observe that 
\[ 
a\frac{\prod_{k=1}^{n}w_k'}{\prod_{k=1}^{n}w_k} \bigg(\prod_{k=1}^{i}w_k \bigg) y_{p,i} -1=a\frac{\prod_{k=1}^{n}w_k'}{\prod_{k=1}^{n}w_k}\bigg(\bigg(\prod_{k=1}^{i}w_k \bigg) y_{p,i}-1\bigg)+\bigg(a\frac{\prod_{k=1}^{n}w_k'}{\prod_{k=1}^{n}w_k}-1\bigg),
\]
which leads to
\[
\left\vert a\frac{\prod_{k=1}^{n}w_k'}{\prod_{k=1}^{n}w_k} \bigg(\prod_{k=1}^{i}w_k \bigg) y_{p,i} -1\right\vert\geq \left\vert a\frac{\prod_{k=1}^{n}w_k'}{\prod_{k=1}^{n}w_k}\bigg(\bigg(\prod_{k=1}^{i}w_k \bigg) y_{p,i}-1\bigg)\right\vert-\left\vert a\frac{\prod_{k=1}^{n}w_k'}{\prod_{k=1}^{n}w_k}-1\right\vert.
\]
By using the inequalities $\vert\big(\prod_{k=1}^{i}w_k\big) y_{p,i}-1\vert>\gamma$ and $\vert a \frac{\prod_{k=1}^{n}w_k'}{\prod_{k=1}^{n}w_k}\vert\geq 1-\left\vert a\frac{\prod_{k=1}^{n}w_k'}{\prod_{k=1}^{n}w_k}-1\right\vert\geq 1-\varepsilon$, we obtain 
\[
\left\vert a\frac{\prod_{k=1}^{n}w_k'}{\prod_{k=1}^{n}w_k} \bigg(\prod_{k=1}^{i}w_k \bigg) y_{p,i} -1\right\vert\geq (1-\varepsilon)\gamma-\varepsilon=\gamma-\varepsilon(1+\gamma),
\]
which gives, since $0<\varepsilon<\frac{\gamma}{2(1+\gamma)}$, 
			\[
			\bigg\vert \frac{\prod_{k=1}^{n}w_k'}{\prod_{k=1}^{n}w_k} \bigg(\prod_{k=1}^{i}w_k\bigg) y_{p,i} -\frac{1}{a}\bigg\vert\geq \frac{\gamma}{2\vert a\vert}.
			\] 
			The proof is complete.
		\end{proof}
		
		The regularity conditions on $A$ in the previous proposition might seem a bit hard to realize. Let us give a concrete example where Proposition \ref{Prop_Two_Shifts} applies.

		\begin{example}
			Let $(\Omega,\mathcal{A},\mathbb{P})$ be a standard probability space and $(X_i)_{i\in\N}$ be an independent and identically distributed (i.i.d.) sequence of real random variables with a Bernoulli distribution. Then, almost surely for every finite sequence $(b_i)_{i=0}^{k}\subset\{0, 1\}$, we have
			\begin{equation}\label{H2}
				\underline{d}(\{n\in\N: \forall \, 0\leq i\leq k, X_{n+i}=b_i\})>0.
			\end{equation}
			Indeed, since the set of finite sequences in $\{0, 1\}$ is countable, it suffices to prove that, for every finite sequence $(b_i)_{i=0}^{k}\subset\{0, 1\}$, we have $\underline{d}(\{n\in\N: \forall \, 0\leq i\leq k, X_{n+i}=b_i\})>0$ almost surely.
			Let $(b_i)_{i=0}^{k}\subset\{0, 1\}$ and for every $n\in\N$ we define $Y_n:=(X_{0+nk},X_{1+nk},\ldots,X_{k-1+nk})$. This is a sequence of i.i.d. random vectors. Moreover, $\P(Y_n=(b_i)_{i=0}^{k})=p^m(1-p)^{n-m}$, where $m$ is the number of ones appearing in $(b_i)_{i=0}^{k}$. We define the random variable $Z_n$ so that $Z_n=1$ if $Y_n=(b_i)_{i=0}^{k}$ and $Z_n=0$ else. Then, $Z_n$ is a Bernoulli variable such that $\P(Z_n=1)=p^m(1-p)^{k-m}$.
			
			Now, an application of the Law of Large Numbers yields 
			\begin{align*}
				\underline{d}(\{n\in\N:(X_i)_{i=n}^{n+k}=(b_i)_{i=0}^{k}\})& \geq \underline{d}(\{n\in\N:Y_n=(b_i)_{i=0}^{k}\})\\
				&=\underline{d}(\{n\in\N:Z_n=1\})\\
				&=\liminf_{n\to+\infty}\frac{\sum_{j=0}^{n}Z_j}{n}\\
				&=p^m(1-p)^{k-m}>0.
			\end{align*}
			For every $\omega\in\Omega$, we define the sequence $A(\omega):=\{n\in\N: X_n(\omega)=0\}$. Let us now construct the shifts $B_w$ and $B_{w'}$.
			Let $w_n=\lambda>1$ for every $n\in\N$. Set $w_1'=\lambda$. Let us choose $\delta>0$, $\varepsilon>0$ and $M>0$ so that $1+\delta<\lambda-\varepsilon<\lambda+\varepsilon<M$. Now we set, for all $n\geq 2$,
			\[
			w'_n=\frac{\lambda^n}{\prod_{i=1}^{n-1}w'_i}\mathds{1}_{\{X_n=0\}}+V_n \mathds{1}_{\{X_n=1\}},
			\] 
			where $(V_n)_{n\in\N}$ is a sequence of random variables taking values in 
			\[\begin{cases}
				(1+\delta, \lambda-\varepsilon]\cup[\lambda+\varepsilon, M), &\text{ if }X_{n-1}=0,\\
				(1+\delta, \lambda-\varepsilon], &\text{ if }X_{n-1}=1\text{ and }V_{n-1}\in(1+\delta, \lambda-\varepsilon],\\
				[\lambda+\varepsilon, M), &\text{ if }X_{n-1}=1\text{ and }V_{n-1}\in[\lambda+\varepsilon, M).
			\end{cases}\]
			First observe that $\sum_{n\in\N}\frac{1}{\prod_{k=1}^{n}w'_k}<+\infty$ and $\sum_{n\in\N}\frac{1}{\prod_{k=1}^{n}w_k}<+\infty$, thus $B_{w'}$ and $B_{w}$ are frequently hypercyclic. Moreover, let us choose $0<\eta<\frac{\varepsilon}{\lambda+\varepsilon}$. By construction we have, for all $n\geq 2$, if $X_n=0$
			\[
			\frac{\prod_{i=1}^{n}w_i}{\prod_{i=1}^{n}w_i'}=1
			\]
			and for all $n$ so that $X_n=1$, observe that 
			\[
			\frac{\prod_{i=1}^{n}w_i}{\prod_{i=1}^{n}w_i'}\geq \frac{\lambda}{\lambda-\varepsilon}\quad \hbox{ or }\quad\frac{\prod_{i=1}^{n}w_i}{\prod_{i=1}^{n}w_i'}\leq \frac{\lambda}{\lambda+\varepsilon}.
			\]
			The choice of $\eta$ ensures that $\frac{\prod_{i=1}^{n}w_i}{\prod_{i=1}^{n}w_i'}\notin [1-\eta,1+\eta]$ if $X_n=1$. Hence, according to Proposition \ref{Prop_Two_Shifts}, there exists $x\in FHC(B_w)$ and $x\notin HC(B_{w'})$.
			
		\end{example}

\section{Concluding remarks and open problems}\label{sec:Conclusion}

In this work, we have established a comprehensive framework of results regarding the existence and non-existence of common vectors that are simultaneously hypercyclic, frequently hypercyclic or upper frequently hypercyclic for a family of bounded linear operators $(T_\lambda)_{\lambda\in\Lambda}$ acting on the same $F$-space $X$ and indexed by a set $\Lambda\subset \R^d$ for some $d\geq 1$. Our findings underscore the role played by the growth rate of the function $F$ associated to the family $(T_\lambda)_{\lambda\in\Lambda}$ and reveal the existence of a critical growth that separates the existence from the non-existence of common vectors. 
Our findings allowed us to deeply investigate the case of more intricate sets of parameters $\Lambda\subset\R^d$, in particular dust-like self-similar fractals. 
By mapping these critical growths, this study provides a clearer landscape of the constraints governing common universality in linear dynamics, illustrating how the ``strength'' of the required dynamical property directly restricts the admissible growth of the function $F$. 
In the following sections, we return to our results and connect them with several open questions.

\subsection{Common hypercyclicity}

In Section \ref{sec:commonhc}, we have shown that the divergence of the series $\sum_{n\in\mathbb{N}} \frac{1}{F(n)}$ is the key component for existence of common vectors when $d=1$. In fact, the generalized Costakis-Sambarino criterion Theorem~\ref{CS:crit:F} allows us to get common hypercyclic vectors no matter how slow the series $\sum_{n\in\mathbb{N}} \frac{1}{F(n)}$ diverges, which let us give examples of more intricate functions than simply $F(n)=n$. 
It is not so clear if we could get a similar result for $d>1$ and general $F$. Existence results in high dimension are more delicate, and without knowing what is $F$ exactly, it seems non-trivial to get such a result in full generality. For instance, \cite[Problem 6.5]{CostaSelf} remains open. We propose the following closely related problem.
\begin{question} \label{Q7:CS:general}
    Can we get a Costakis-Sambarino criterion for sets of parameters $\Lambda\subset \R^2$ by assuming the divergence of the series $\sum_{n\in\N}\frac{1}{F(n)^2}$?
\end{question}
On the other way around, we have demonstrated in Section~\ref{sec:hc:diverge} that, at least for shift operators, the convergence of $\sum_{n\in\mathbb{N}} \frac{1}{F(n)}$ acts as a barrier to the existence of common vectors when $d=1$. More generally, in Theorem \ref{thm:F:conv:non-ex} we have shown that, for all $d\geq 1$ and for any set of parameters $\Lambda\subset \R^d$ with Hausdorff dimension $s\in(0,d]$, the divergence of the series $\sum_{n\in\N} \frac{1}{F(n)^s}$ is a necessary condition for the existence of common hypercyclic vectors for weighted shifts operators. We wonder if the same is true for general families, but it is not immediate to find an $F$ that would serve as the associated function to the family $(T_\lambda)_{\lambda\in \Lambda}$. For instance, if we interpret $F$ as the function $\psi$ appearing in \cite[Theorem 3.1]{BCMparam}, then the referred result says exactly what we wants: the divergence of the series $\sum_{n\in\N} \frac{1}{F(n)^s}$ is a necessary condition for the existence of common hypercyclic vectors when $\Lambda$ is a $s$-dimensional set of parameter with $\H^s(\Lambda)>0$. However, if we want to think as in this article, then a more natural approach would be to consider the following setting:
\begin{itemize}
    \item $\Lambda\subset\R^d$ has Hausdorff dimension $s\in (0,d]$;
    \item $\H^s(\Lambda)>0$;
    \item $F$ is the function satisfying equation \eqref{cond:comp-2} (or simply \eqref{cond:comp-3}).
    \end{itemize}

\begin{question}\label{Q:sumF:fini}
    
    Can we prove that $\sum 1/F(n)^s<+\infty$ implies the non-existence of common hypercyclic vectors for $(T_\lambda)_{\lambda\in\Lambda}$?
\end{question}

Finally let us return to Example \ref{diago}. We have demonstrated that, for the weight family $w(a,b)=(w_n(a,b))$ satisfying $w_1(a,b) \dots w_n(a,b) = \exp(an^{b})$ as considered in \cite{BCMparam}, the collection $(B_{w(a,a)})_{a\in(0,1]}$ possesses a dense $G_\delta$-set of common hypercyclic vectors. However, the following general question remains open.

\begin{question}
Does the family $(B_{w(a,b)})_{(a,b)\in(0,1]^2}$ share a common hypercyclic vector?
\end{question}

\subsection{Common upper frequent hypercyclicity}

The notions of hypercyclicity and upper frequently hypercyclicity share some key features. For instance, for an operator $T$, both $HC(T)$ and $\U FHC(T)$ are either empty or residual. When it comes to families $(T_\lambda)_{\lambda\in\Lambda}$ of operators and $d=1$, the divergence of the series $\sum_{n\in\N}\frac{1}{F(n)}$ implies the existence of hypercyclic vectors. For upper common hypercyclicity, something similar happens. From Theorem \ref{monia:thm}, what implies the existence of common upper frequent hypercyclic vectors is the divergence of the series $\sum_{n\in\N} \frac{1}{F(\gamma^n)}$ for all $\gamma\in\N$. This is satisfied whenever $F$ grows as $\log$ or slower. On the other way around, the family $\Psi^*$ defined in Notation \ref{notation:F} provides a large family of functions $F\in \Psi^*$, for which we managed to provide non-existence results (see Theorem~\ref{modif_teo:non-ex} and thereafter). Apart from some natural properties that these $F$ are requested to have, a particularly interesting one, in this $d=1$ framework, is the convergence of the series $\sum_{n\in\N} \frac{1}{F(\gamma^n)}$ for some $\gamma\in\N$. As mentioned in Remark \ref{remark:neg:monia}, this is precisely the opposite of what the existence Theorem \ref{monia:thm} requires, thus our non-existence results are optimal from this point of view when $d=1$. In Section \ref{subsec:common:ufhc:neg} we were able to exhibit functions $F$ that grow slightly faster than $\log$ for which we have no common upper frequent hypercyclic vector (see Example \ref{example:prod:comp:log:ufhc}). 
In our study, we have obtained non-existence results for families $(T_\lambda)_{\lambda\in \Lambda}$ indexed by subsets $\Lambda$ of $\R^d$, $d\geq 1$. On the existence side, however, Theorem \ref{monia:thm} is the only general result in the literature. The following question is very natural.

\begin{question}
    Can we generalize Theorem \ref{monia:thm} to include families indexed by a square $\Lambda = [a,b]^2$? Can we do the same for more general $\Lambda\subset \R^d$, $d\geq 1$?
\end{question}

These ``more general sets $\Lambda\subset \R^d$'' should be understood as the ones considered along Section~\ref{sec:Common((U)F)HC:negative}. 
Indeed, our methods allows to apply our results to parameter sets that are self-similar fractals satisfying the Open Set Condition. Thanks to this flexibility, we are able to apply our non-existence results on families whose parameter sets can be not only intervals, but also many other fractals in several dimensions. In particular, we proved that multiples of the backward shift do not admit common upper frequently hypercyclic vectors when the parameter set is a self-similar fractal like the Cantor ternary set. However, we did not study the case of self-similar sets constructed with contraction ratios that change at each stage of the construction (like the fat Cantor set, for example). One may also wonder what remains if we remove the Open Set Condition. Indeed, this condition is probably not necessary and a weaker condition might suffice.

\begin{question}
    Can these results be extended to more general self-similar sets? What kind of weaker condition could replace the Open Set Condition?
\end{question}

The notion of hypercyclicity can be viewed as a form of \textit{$(\alpha_k)$-upper frequent hypercyclicity} with respect to the upper asymptotic density associated with the weights $\alpha_k = e^k$, $k \geq 1$. Such a weighted density $\overline{d}_{(\alpha_k)}$ is defined as follows: for any $E \subset \mathbb{N}$, 
\[
\overline{d}_{(\alpha_k)}(E) = \limsup_{n \to +\infty} \frac{\sum_{k=1}^n \alpha_k\mathds{1}_E(k)}{\sum_{k=1}^n \alpha_k},
\]
where $(\alpha_k)_{k\in\N}$ is a sequence of non-negative real numbers such that $\sum_{k=1}^n\alpha_k\to+\infty$ as $n$ tends to infinity (for further details, we refer the reader to \cite{ErnMou,ErnMou2}). In this framework, the classical notion of upper frequent hypercyclicity corresponds to the weight sequence $\alpha_k = 1$, $k \geq 1$. The following question then arises naturally.
\begin{question}\label{questionpoids}
    What is the relationship between the growth of the function $F$ and the existence or non-existence of common $(\alpha_k)$-upper frequent hypercyclic vectors for a family $(T_{\lambda})$, specifically when $\alpha_k \to +\infty$ and $\alpha_k e^{-k} \to 0$? 
\end{question}
\noindent These last conditions yield intermediate notions between upper frequent hypercyclicity and standard hypercyclicity.

\subsection{Common frequent hypercyclicity} 
This property represents the most restrictive case, since the set of frequently hypercyclic vectors for an operator is always a set of Baire first category. Moreover, it is known that common frequent hypercyclicity is a notion that cannot be achieved for uncountably many multiples of a single operator. A criterion for the existence of common frequently hypercyclic vectors for a countable family of operators was also recently established. Here, by following the aforementioned approach that associates a ``Lipschitz constant'' function $F$ with a family of operators, we have established a general existence result for common frequently hypercyclic vectors for an uncountable family of operators. And once again, the growth of this function determines the existence or non-existence of common frequently hypercyclic vectors. Our results in Theorem~\ref{thmgeneral} highlight that the existence of common vectors essentially requires $(F(n))_{n\in\mathbb{N}}$ to be bounded, which is the same condition as being constant up to replacing $F$ by its supremum. In contrast, the non-existence result in Theorem~\ref{teo:non-exFHC} is triggered as soon as $(F(n))_{n\in\mathbb{N}}$ tends to infinity.

Although Theorem~\ref{thmgeneral} applies to general families of operators, we illustrated its usefulness only in the case of weighted shifts. It would be interesting to find more diversified applications. 

\begin{question}
    Are there natural applications of Theorem~\ref{thmgeneral} that do not involve shift operators?
\end{question}

As in Question \ref{questionpoids}, the notion of frequent hypercyclicity can be viewed as a form of \textit{$(\alpha_k)$-frequent hypercyclicity} with respect to the lower asymptotic density associated with the weights $\alpha_k = 1$, $k \geq 1$. These weighted densities, denoted by $\underline{d}_{(\alpha_k)}$, are defined for any $E \subset \mathbb{N}$ as:
\[
\underline{d}_{(\alpha_k)}(E) = \liminf_{n \to +\infty} \frac{\sum_{k=1}^n \alpha_k\mathds{1}_E(k)}{\sum_{k=1}^n \alpha_k}, 
\]
where $(\alpha_k)_{k\in\N}$ is a sequence of non-negative real numbers such that $\sum_{k=1}^n\alpha_k\to+\infty$ as $n$ tends to infinity. Notice that, for any $E \subset \mathbb{N}$, $\overline{d}_{(\alpha_k)}(E) = 1-\underline{d}_{(\alpha_k)}(\mathbb{N}\setminus E)$. 
The following question arises once again.

\begin{question}
 What is the connection between the growth of the function $F$ and the existence or non-existence of common $(\alpha_k)$-frequent hypercyclic vectors for a family $(T_{\lambda})$? 
\end{question}
\noindent This last question is only relevant when $\alpha_k e^{-k} \to 0$, as it is known that $(e^k)$-frequently hypercyclic operators do not exist (see \cite[Proposition 3.11]{ErnMou2}).

\section*{Funding}

The first author was partially supported by Grants No. 406457/2023-9, No. 403964/2024-5 and No. 304165/2025-5, all from the Conselho Nacional de Desenvolvimento Científico e Tecnológico (CNPq, Brazil), and by the Université du Littoral Côte d'Opale (ULCO) during his research stay in Calais, France. The second author was partially supported by the Brazilian-French Network in Mathematics. The fourth author was partially supported by the grant ANR-24-CE40-0892-01 of the French National Research Agency ANR (project ComOp).

\section*{Acknowledgments}

This work was partially developed during the first author's visit to the Laboratoire de Mathématiques Pures et Appliquées Joseph Liouville at the Université du Littoral Côte d'Opale (ULCO), in Calais, France, in May 2025. The first author gratefully acknowledges the warm hospitality of the laboratory team and the support provided by ULCO during his stay. It was also partially developed during the second author's research stay in the Departamento de Matemática at Universidade Federal da Paraíba (UFPB), in João Pessoa, Brazil, in July 2026. The second author gratefully acknowledges the warm hospitality of the research team during his stay.

\bibliographystyle{abbrv}
\bibliography{bib.bib}
    
\end{document}